\documentclass{amsart}
\usepackage{amssymb}
\usepackage{graphics}

\usepackage{xy}
\xyoption{all}
\theoremstyle{plain}

\numberwithin{equation}{section}

\newtheorem{theorem}{Theorem}
\newtheorem{theoremapp}{Theorem}

\newtheorem{proposition}[equation]{Proposition}

\newtheorem{lemma}[equation]{Lemma}
\newtheorem{assumptions}[equation]{Assumption}

\theoremstyle{definition}
\newtheorem{remark}[equation]{Remark}
\newtheorem*{example}{Example}
\newtheorem{definition}[equation]{Definition}
\newtheorem*{question}{Question}

\newcommand{\Pro}{\mathbb{P}^{\infty}}

\newcommand {\printname}[1] {}

\newcommand{\remove}[1]{}

\def    \R  {{\Bbb R}}
\def    \Z  {{\Bbb Z}}
\def    \Q  {{\Bbb Q}}

\def    \CP {{\Bbb {CP}}}
\def    \P {{\Bbb {P}}}
\def    \C  {{\Bbb C}}
\def    \N  {{\Bbb N}}
\def   \S   {{\Bbb S}}

\def    \Tilde  {\widetilde}

\def    \Gt     {\Tilde{G}}

\begin{document}
\title[One-connectivity  and finiteness of  Hamiltonian $S^1$-manifolds]{One-connectivity  and  finiteness of  Hamiltonian $S^1$-manifolds  with minimal fixed sets}

\author{Hui Li, Martin Olbermann and Donald  Stanley }
\address{School of Mathematical Sciences\\
Soochow University\\
Suzhou, 215006\\
China.}
\email{hui.li@suda.edu.cn}

\address{Ruhr-Universit\"at Bochum\\
        Fakult\"at f\"ur Mathematik\\
        44780 Bochum, Germany.}
\email{martin.olbermann@rub.de}

\address{University of Regina\\
Department of Mathematics and Statistics\\
College West 307.14\\
Regina, Saskatchewan\\
Canada S4S OA2.}
\email{stanley@math.uregina.ca}

\thanks{2010 classification. 53D05,   57R65,  53D20,
55Q05, 57R80.}
\begin{abstract}
Let the circle act effectively in a Hamiltonian fashion on a compact symplectic
manifold $(M, \omega)$. Assume that the fixed point set $M^{S^1}$
has exactly two components, $X$ and $Y$, and that  $\dim(X) +
\dim(Y) +2 = \dim(M)$.  We first show that $X$, $Y$ and $M$ are
simply connected.  Then we show that, up to $S^1$-equivariant diffeomorphism, there are
finitely many such manifolds in each dimension. Moreover, we show that in low dimensions, the manifold
is unique in a certain category.  We use techniques from  both areas of symplectic geometry and geometric
topology.
\end{abstract}

 \maketitle

\section{Introduction}
Let a Lie group act non-trivially on a manifold $M$.  The
manifold $M$ may have certain geometric or topological structures
which are invariant under the action. For instance, $M$ may
have a symplectic structure  or a K\"ahler structure, $M$ may have
a certain homotopy type or a certain cohomology ring.  When $M$ is
a symplectic manifold,  and when the Lie group action is Hamiltonian,
the moment map provides a key tool for the
study of $M$.

  One fundamental question in symplectic geometry is  which symplectic manifold admits a Hamiltonian group action,  or,  when  a symplectic group action
  is a Hamiltonian action. This has been answered in various cases.
  One way of approaching the question  is to look at necessary
  conditions,  and to find the topological and geometrical properties
  such manifolds can have. Furthermore, we can ask how many such manifolds can
  arise with these given properties. These can be the first steps for further
  classifications of the manifolds.

 Consider a compact  connected symplectic
manifold $(M, \omega)$ which admits an effective Hamiltonian $T$-action, where $T$ is a connected compact torus. 
The number $k =\frac{1}{2}\dim(M) -\dim (T)$ is called the {\em complexity} of the
Hamiltonian $T$-manifold $M$.  Complexity zero Hamiltonian manifolds,  also called symplectic toric manifolds,  are classified \cite{D}. Complexity one Hamiltonian $4$-manifolds and complexity one Hamiltonian manifolds whose nonempty symplectic quotients are all $2$-dimensional  are also classified \cite{{AH}, {Au1}, {Kar}, {KT1}, {KT2}}.  All symplectic toric manifolds are K\"ahler \cite{D},
while there exist complexity one Hamiltonian manifolds which do not admit any invariant K\"ahler structure \cite{T0}. 

Let $(M, \omega)$ be a compact symplectic manifold which admits an effective Hamiltonian $T=S^1$-action with moment map  $\phi\colon M\to \R$.  A general classification of such manifolds does not appear tractable.
Let us first look at a simple constraint on the dimensions of the connected components of the fixed point set $M^{S^1}$. First of all, each connected component of $M^{S^1}$ is symplectic hence even dimensional.
The moment map $\phi$ is a perfect Morse-Bott function whose critical set is exactly $M^{S^1}$. The fact that $M$ is compact and symplectic implies that $b_{2i}(M)\geq 1$ for all $0\leq 2i\leq
\dim(M)$.  Using Morse-Bott theory,  we may obtain the inequality
$$\sum_{F\subset M^{S^1}} \left( \dim(F)  + 2 \right)
\geq\dim(M) + 2,$$
 where the sum is over all the connected components of $M^{S^1}$.  When $M$ has minimal even Betti numbers, i.e.,  when  $b_{2i}(M)= b_{2i}(\CP^n) = 1$ for all $0\leq 2i\leq 2n =\dim(M)$,  we have the equality
\begin{equation}\label{dim}
\sum_{F\subset M^{S^1}} \left( \dim(F)  + 2 \right)
=\dim(M) + 2.
\end{equation}
 But this equality does not imply that the even Betti
numbers of $M$ are minimal. For details of these arguments, we refer
to \cite[Section 4]{LT}. Note that the critical set of the moment map has at least two connected components
 --- its minimum and its maximum.
 
 Recent works on compact Hamiltonian $S^1$-manifolds  which have
minimal even Betti numbers or which satisfy (\ref{dim}) include \cite{{To}, {Mc}, {LT},  {L1}, {GS}, {L2}}. In these works,
the authors consider various cases, for which they show that certain important global invariants --- the integral cohomology ring and total Chern class, of the manifold, as well as the circle representations on the normal bundles of the fixed components,  are identical to those of some known examples. In several  interesting cases, e.g.~when the manifold is $6$-dimensional, or is  K\"ahler,  the manifold can be equivariantly identified in the symplectic or complex categories with some standard examples \cite{{Mc}, {L1}, {L2}}. 
Similarly, in another recent  work \cite{Mo}, the author studies the integral cohomology ring and total Chern class of
compact Hamiltonian $T$-manifolds which are GKM-manifolds  with minimal even Betti numbers.

  There is also the recent work \cite{HH}, where the authors study compact Hamiltonian $S^1$-manifolds 
whose fixed point set $M^{S^1}$ consists of two connected components (but without (\ref{dim})). They give certain description
on when two such manifolds are equivariantly diffeomorphic.

\smallskip

In this paper, we restrict our attention to the study of those manifolds satisfying the following
assumption:
\begin{assumptions}\label{conditions}
Let  $(M, \omega)$ be a $2n$-dimensional compact symplectic manifold equipped with an effective Hamiltonian $S^1$-action such that the fixed point set
$M^{S^1}$ consists of two connected components, $X$ and $Y$, satisfying (\ref{dim}), which we rewrite as
 \begin{equation}\label{dim=}
\dim(X) + \dim(Y) + 2 = \dim(M).
\end{equation}
\end{assumptions}

In Section~\ref{section:outline}, we
will see that the K\"ahler manifolds $\CP^n$, and $\Gt_2(\R^{n+2})$ --- the Grassmannian of oriented two-planes in
$\R^{n+2}$ with $n\geq 3$ odd, equipped with some standard circle actions, provide
two families of standard examples of manifolds satisfying Assumption \ref{conditions}. 

 Suppose Assumption \ref{conditions} holds. 
In \cite{LT}, the first author and Tolman
determine the integral cohomology rings and total Chern classes of $X$, $Y$ and $M$. In particular, they prove that the integral cohomology ring and total Chern class of $M$ are isomorphic to either those of $\CP^n$, or those of
$\Gt_2(\R^{n+2})$ with $n\geq 3$ odd. In this paper,  we 
determine the fundamental groups of $X$, $Y$ and $M$, and using this, we show that  in the equivariant smooth category,  there exist only finitely many such manifolds. Hence in a sense, our results show that the manifolds satisfying Assumption \ref{conditions} 
are ``very close to"  the standard examples $\CP^n$ and  $\Gt_2(\R^{n+2})$ with $n$ odd. 
We use both symplectic and topological techniques, in particular, we use surgery theory for the proof of the 
 finiteness. We hope our results and method will provide insights
for further classification of compact Hamiltonian $S^1$-manifolds which have
minimal even Betti numbers or which satisfy (\ref{dim}).

 \smallskip

 Now, let us state more precisely our results. 

\begin{theorem}\label{A}
Under Assumption \ref{conditions}, the manifolds $M$, $X$ and $Y$ are all $1$-connected.
\end{theorem}

\begin{theorem}\label{hcp}
If Assumption \ref{conditions} holds, then
\begin{enumerate}
\item both $X$ and $Y$ are homotopy complex projective spaces with
standard Pontryagin classes;
\item
 when the action is {\em semifree}, i.e., the action is free outside the fixed point set, then $M$
is a homotopy complex projective space with standard Pontryagin
classes.
\end{enumerate}
\end{theorem}

The study of
Hamiltonian $S^1$-manifolds with minimal even Betti numbers was also
motivated by the classical Petrie's conjecture, see \cite{To}.
The conjecture states that if a $2n$-dimensional manifold $M$ of the homotopy type of $\CP^n$
admits a non-trivial circle action, then the total Pontryagin
class of $M$ agrees with that of $\CP^n$.  It is known that Petrie's
conjecture holds when the fixed point set $M^{S^1}$ 
consists of  a small number of connected components, in particular, $2$ connected components \cite{{Wa1}, {Yo}}.
Theorem~\ref{hcp} (2) tells us that when $M$ is a compact symplectic manifold with
a semi-free, Hamiltonian circle action whose fixed point set consists of two connected components satisfying (\ref{dim=}), we have both the assumption and the conclusion of Petrie's conjecture.

 \smallskip

 Our next result says that in the equivariant smooth category, for each fixed dimension, there are only finitely many manifolds  satisfying Assumption \ref{conditions}.

\begin{theorem}\label{equiv.diffeo}
In each fixed dimension, up to $S^1$-equivariant diffeomorphism, there are only finitely many  manifolds fulfilling Assumption \ref{conditions}.
\end{theorem}

 In fact, under Assumption \ref{conditions},  if $M$ or one of $X$ and $Y$ is low-dimensional, we can determine  $M$ in the equivariant symplectic category.  When $X$ or $Y$ is an isolated point, by a theorem of Delzant \cite{D}, $M$ is $S^1$-equivariantly symplectomorphic to $\CP^n$ with a standard circle action. In particular, when $\dim(M) = 2$ or $4$, $X$ or $Y$ has to be isolated. When $\dim(M) = 6$, then either $X$ or $Y$ is an isolated point, or both  $X$ and $Y$ are $2$-dimensional --- $2$-spheres by Theorem~\ref{A}. In the latter case, if the $S^1$ action on $M$ is semifree,  by Gonzalez's work  \cite{G},  $M$ is $S^1$-equivariantly symplectomorphic to $\CP^3$ with a standard circle action. See Section~\ref{discussion} for more discussion.

\begin{remark}\label{lowdim}
For the manifolds fulfilling Assumption \ref{conditions},
we have the following non-equivariant uniqueness in low dimensions:
\begin{enumerate}
\item  when $\dim M = 2n \leq 6$, $M$ is
diffeomorphic to $\CP^n$ if the action is semifree, and is
diffeomorphic to $\Gt_2(\R^{5})$ if the action is not semifree;
\item when $\dim M=10$ or $\dim M=14$ and the action is not semifree,
 $M$ is respectively homeomorphic to $\Gt_2(\R^{7})$ or to $\Gt_2(\R^{9})$.
\end{enumerate}
\end{remark}

\subsubsection*{Acknowledgement}  The work of this paper was done while the authors were visiting
the Max-Planck-Institute. We thank the institute for the fellowship. We thank Diarmuid Crowley and James Davis for many
useful discussions and remarks and for suggesting the reference
\cite{FW}. We thank Ron Fintushel for providing us with a theorem about symplectic homotopy complex projective planes. Finally, we thank the referee for many suggestions and comments which helped us to improve the exposition.

\section{On our results}\label{section:outline}
Let us first give two families of examples of manifolds satisfying Assumption \ref{conditions}, then state earlier results on such manifolds
and finally give a short outline on our proofs.

\begin{example}
[A] Given $n \geq 1$, let $\CP^n$ be the complex projective
space. It naturally arises as a coadjoint orbit of $SU(n+1)$; it
inherits a symplectic form $\omega$ and a Hamiltonian $SU(n+1)$
action.

For any $j \in \{0,\dots,n-1\}$, there is a semifree Hamiltonian
circle action given by
\begin{equation}\label{standard}
\lambda\cdot [z_0, z_1, ..., z_n]=[\lambda z_0, \lambda z_1, ...,
\lambda z_j, z_{j+1}, ..., z_{n}].
\end{equation}
The fixed set consists of two components:
$$\left\{ [z] \in \CP^n
\mid z_k = 0 \ \forall \ k  \leq j \right\}
\simeq \CP^{n-j-1},
\,\,  \mbox{and} \,\,
\left\{ [z] \in \CP^n
\mid z_k = 0 \ \forall \ k  > j \right\}
\simeq \CP^{j}$$
which satisfy (\ref{dim=}).
\end{example}

\begin{example}
[B] Given $n \geq 3$, let  $\Tilde G_2(\R^{n+2})$ denote the
Grassmannian of oriented two-planes in $\R^{n+2}$. This
$2n$-dimensional manifold naturally arises as a coadjoint orbit of
$SO(n+2)$; it inherits a symplectic form $\omega$ and a Hamiltonian
$SO(n+2)$ action.

If $n$ is odd, there is a Hamiltonian circle action on $\Tilde
G_2(\R^{n+2})$ induced by the action on $\R^{n+2}\cong \R \times
\C^{\frac{1}{2}(n+1)}$ given by
$$\lambda \cdot (t, z_1, ..., z_{\frac{1}{2}(n+1)})=
(t, \lambda z_1, ..., \lambda z_{\frac{1}{2}(n+1)}).$$ The fixed set
consists of two components, corresponding to the two orientations on
those real two--planes which are complex lines in $ \{0\} \times \C^{\frac{1}{2}(n+1)}$.
Hence the fixed point set consists of two copies of  $\P \big( \{0\} \times \C^{\frac{1}{2}(n+1)}
\big) \simeq \CP^{\frac{1}{2}(n-1)}$. Moreover, the set of two--planes which lie entirely in $\{0\} \times \C^{\frac{1}{2}(n+1)}$ is
fixed by $\Z_2$. This submanifold, which is symplectomorphic to
$\Tilde G_2(\R^{n+1})$, has codimension $2$ in $M$.
\end{example}

The following theorem summarizes the previous results about those manifolds satisfying Assumption \ref{conditions}.
Note that the cohomology ring and Chern classes  of $M$ in (\ref{ma}) are  the same as those of $\CP^n$,  and the cohomology ring and Chern classes of $M$
 in (\ref{mb}) are the same as those of $\Tilde G_2(\R^{n+2})$, where $n\geq 3$ is odd.

\begin{theorem}\label{LT}\cite[Theorems 1 and 2]{LT}
If Assumption \ref{conditions} holds, then:
\begin{gather}
H^*(X;\Z) = \Z[u]/u^{i  + 1} \ \  \mbox{and} \ \  c(X) =
(1+u)^{i+1},
\quad \mbox{where} \  \dim(X) = 2i; \label{hcx} \\
H^*(Y;\Z) = \Z[v]/v^{j + 1} \ \  \mbox{and} \ \  c(Y) = (1+v)^{j+1},
\quad \mbox{where} \  \dim(Y) = 2j. \label{hcy}
\end{gather}
Moreover, we only have two cases {\bf (A)} and {\bf (B)} below,
where $\deg(x) =2$ and $\deg(y) = n+1$.

\noindent{\bf (A)} the action is semifree,
\begin{equation}\label{ma}
H^*(M; \Z)=\Z[x]/(x^{n + 1}) \ \ \mbox{and} \ \ c(M)=(1+x)^{n + 1};
\end{equation}
\noindent {\bf (B)}  the  action is not semifree, $\dim(X)=\dim(Y)$,
the only finite nontrivial stabilizer group  is $\Z_2$ and  the submanifold
fixed by $\Z_2$ has codimension $2$ in $M$.  In this case,  $n\geq
3$ is odd,
\begin{equation}\label{mb}
 H^*(M;\Z) =
\Z[x,y]/\big(x^{\frac{1}{2}(n+1) } - 2y, y^2\big) \ \ \mbox{and} \ \
c(M) = \frac {(1+x)^{n + 2}} {1+2x}.
\end{equation}
In both cases, the Chern classes of each normal subbundle of $X$ and of
$Y$ on which $S^1$ acts with a fixed stabilizer group are completely determined.
\end{theorem}

Now, we give a short outline of the proof of our results.

\medskip

We prove Theorem~\ref{A} in Section~\ref{pi1thmsection}. We first
show that $\pi_1(X)=\pi_1(Y)=1$, and then a simple Seifert-van
Kampen argument  (or alternatively the main theorem of \cite{L}) 
shows that $\pi_1(M)=1$.

To prove  $\pi_1(Y)=1$, we consider the symplectic reduced space
$M_r = \phi^{-1}(r)/S^1$ at any regular value $r$ of the moment map $\phi$. The space $M_r$
is a weighted complex projective bundle over both $X$ and $Y$ at the
same time. We show that the weighted complex projective spaces that
actually arise are  homeomorphic to a smooth complex projective
space. We then show that the inclusion of the fiber of $M_r$ as a
bundle over $X$ composed with the projection down to $Y$ induces an
isomorphism on top cohomology groups, and this allows us to obtain
$\pi_1(Y)=1$. Similarly $\pi_1(X)=1$.\\

The short Section~\ref{corolsection} contains the proof of Theorem~\ref{hcp}
which is a consequence of Theorem~\ref{A} and (\ref{hcx}), (\ref{hcy}), (\ref{ma}).\\

We prove Theorem~\ref{equiv.diffeo} in Section~\ref{eqthmsection}.
The idea is to glue a tubular neighborhood of $X$ and a tubular
neighborhood of $Y$ along a regular level set of the moment map.
We prove that the equivariant diffeomorphism
types of the tubular neighborhoods are determined up to finite
ambiguity, and that there are finitely many essentially different
ways of gluing. In particular, in the proof of the latter,  we use the surgery exact sequence and
related techniques, and here we need $\dim(M) > 6$, so we need to treat the 6-dimensional case separately. For the case
when $\dim(M) > 6$ and the action is semifree, one can adopt K. Wang's proof \cite{Wa1} which uses gluing
along the smooth quotient of a regular level set of the moment map.
We will give a proof for the case when  $\dim(M) > 6$ and the action is not semifree
using equivariant gluing along the level set itself; and we give a proof for the case when $\dim(M)=6$.\\

In the Appendix, we prove the uniqueness results mentioned in Remark~\ref{lowdim}.
 As we have mentioned in the Introduction, when $\dim(M)=
 2$ or $4$, $M$ is respectively equivariantly symplectomorphic to
$\CP^1$ and to $\CP^2$. When $\dim(M)= 6$,  either one of $X$ and
$Y$ is  isolated or both $X$ and $Y$ are $2$-spheres; when
 the action is semifree, by Delzant's theorem and by Gonzalez's result,  $M$ is
 equivariantly symplectomorphic to $\CP^3$ with a standard action
 (\ref{standard}).
 For the case of non-semifree actions, we use Kreck's modified surgery
theory and a computation of the relevant bordism group by F. Fang and J. Wang.

\section{Proof of Theorem~\ref{A}}\label{pi1thmsection}
\subsection{Degree one maps from a simply connected topological manifold}
\
\smallskip

The proof of Theorem~\ref{A} will use the following result which is due to Olum.
\begin{lemma}\label{orc} (Olum)
Let $P$ and $X$ be closed oriented topological $n$-manifolds, and
assume that $P$ is simply connected. If there exists a map $f\colon
P\rightarrow X$ such that $f^*\colon H^{n}(X; \Z)\to H^{n}(P; \Z)$
is an isomorphism, then $\pi_1(X) = 1$.
\end{lemma}

The proof of Lemma~\ref{orc} uses the following lemma whose proof we omit here --- it is a routine consequence of standard results about
fundamental classes (see for example \cite[Sections 16.3 and 16.4]{tD}).
\begin{lemma}\label{rorc}
Let $X$ be a closed orientable topological $n$-manifold and $\pi\colon \Tilde
X\rightarrow X$ the universal covering map. If $\pi_1(X)$ is
infinite then $H^n(\Tilde X;\Z)\otimes {\mathbb Q}=0$.
If $\pi_1(X)$ is finite then $\pi^*
\colon H^n(X;\Z)\rightarrow H^n(\Tilde X;\Z)$ is a degree $\pm
\,|\pi_1(X)|$ map.
\end{lemma}

\begin{proof}[Proof of Lemma~\ref{orc}]
Since $P$ is simply connected, the map $f$ factors through the
universal covering map $\pi\colon \Tilde X \rightarrow X$, giving us
a commutative diagram:
$$
\xymatrix
{
& \Tilde X \ar[d]^{\pi}\\
 P \ar[r]_f \ar[ur]^{\Tilde f} & X. }
$$
Taking $H^{n}$ gives
$$
\xymatrix
{
& H^{n}(\Tilde X; \Z) \ar[dl]_{\Tilde f^*}\\
\Z \cong H^{n}(P; \Z) & H^{n}(X; \Z)\cong\Z \ar[u]_{\pi^*}
\ar[l]^{f^*}, }
$$
i.e.,  $f^*=\Tilde f^*\pi^*$.

If $\pi_1(X)$ is infinite then by Lemma~\ref{rorc}, $H^{n}(\Tilde
X;\Z)\otimes {\mathbb Q} =0$ and so
$f^*\otimes{\mathbb Q}=0$, giving a contradiction. So we know that
$\pi_1(X)$ is finite and then again by Lemma~\ref{rorc}, and using
the isomorphisms of  $H^{n}(P; \Z)$  and $H^{n}(X; \Z)$ with $\Z$,
we get
$$
f^* (1) = \Tilde f^* \pi^* (1) = \Tilde
f^*\left(\pm\,|\pi_1(X)|\right) = \pm\, |\pi_1(X)|.
$$
Since $f^*$ is an isomorphism,  $|\pi_1(X)|=1$, hence, $\pi_1(X) =1$.
\end{proof}

\subsection{Weighted projective spaces}
\
\medskip

Let $S^1$ act on $\C^{k+1}$  with  weight vector $w= (w_1,
\dots, w_{k+1})$, where $w_i\in \N$ for all $i$, i.e., let $S^1$ act
by
\begin{equation}\label{eqaction}
\lambda\cdot (z_1, \dots, z_{k+1}) = (\lambda^{w_1}z_1, \dots,
\lambda^{w_{k+1}} z_{k+1})
\end{equation}
for each $\lambda\in S^1$. This action preserves the unit sphere
$S^{2k+1}\subset \C^{k+1}$. The orbifold $\CP^k_w = S^{2k+1}/S^1$ is
called a {\em weighted projective space}.

When $w=(1, \dots, 1)$, $\CP^k_w$ is the smooth projective space
$\CP^k$.

\begin{lemma}\label{221}
The pair consisting of the weighted projective space $\CP^k_{w}$ with $w=(2,\dots, 2, 1)$ and its subspace of singular points
is homeomorphic to the pair $(\CP^k,\CP^{k-1})$ consisting of the smooth projective space and a projective hyperplane.
\end{lemma}

\begin{proof}
We define an equivariant map $S^{2k+1} \to S^{2k+1}$, where the
$S^1$-action on the first sphere has weight vector $w=(2,\dots, 2,
1)$, and the $S^1$-action on the second sphere has weight vector
$w=(2,\dots, 2, 2)$, by the formula
$$(z_1,\dots z_k,z_{k+1})\mapsto \left(z_1,\dots, z_k,\frac{z_{k+1}^2}{|z_{k+1}|}\right).$$
It is easy to check that this map induces a homeomorphism on the
quotient spaces. It maps the singular subspace of $\CP^k_{w}$, that is, all points with vanishing last coordinate,
to the hyperplane in $\CP^k$ given as the vanishing set of the last coordinate.
\end{proof}

\begin{remark}
In general, a weighted projective space $\CP^k_{w}$ with $k > 1$ and
weight vector $w$ may not be homeomorphic to a smooth projective
space. Its integral cohomology may have a twisted ring structure, see \cite{Ka}.
Here and elsewhere in this paper, the cohomology of a weighted projective space is
its singular cohomology as a topological space.
\end{remark}

\subsection{The Duistermaat--Heckman theorem and the Euler class of circle bundles}
\
\medskip

In this section, we  recall the Duistermaat--Heckman Theorem for
the case of Hamiltonian circle actions, and we  address the Euler
class of circle bundles. In particular we explain the
Duistermaat--Heckman Theorem when the circle action has
non-trivial
finite stabilizer groups and obtain Lemma~\ref{wpj}.\\

Let $\Lambda\subset \R$ be any lattice. An $\R/\Lambda$-principal
bundle over a topological space $M$ is described by a \v{C}ech
cocycle which is a continuous $\R/\Lambda$-valued function, and its
Euler class is the image under the isomorphism
$\check{H}^1\left(M;\underline{\R/\Lambda}\right)\cong H^2(M;\Lambda)$.
Sometimes, we will consider the image of the Euler class under the
natural map $H^2(M; \Lambda)\to H^2(M; \R)$ induced by the inclusion
$\Lambda\hookrightarrow \R$.

Let $S^1=\R/\Z$ act on a connected symplectic manifold $(M, \omega)$
with proper moment map $\phi\colon M\to \R$.  Let
$r\in\mbox{image}(\phi)$ be a fixed regular value. Let $I$ be a
connected open interval of regular values of $\phi$ such that $r\in
I$. Then $\phi^{-1}(I)$ is $S^1$-equivariantly diffeomorphic to
$\phi^{-1} (r)\times I$ (in the latter, the moment map is the
projection to $I$). Hence, for any value $a\in I$, $\phi^{-1} (a)$
is $S^1$-equivariantly diffeomorphic to $\phi^{-1} (r)$, and so the
symplectic reduced space $M_a = \phi^{-1} (a)/ S^1$ (possibly a
symplectic orbifold) is diffeomorphic to $M_r = \phi^{-1} (r)/ S^1$.
This gives us a way to identify the cohomology groups of $M_a$ with those
of $M_r$. Since two such trivializations induce homotopic
diffeomorphisms, the identification of $H^*(M_a)$ and of $H^*(M_r)$
is canonical.

The finite stabilizer groups on each $\phi^{-1} (a)$, $a\in I$ are the same as those on $\phi^{-1} (r)$. Let
$m$ be the least common multiple of the orders of all the finite stabilizer groups on $\phi^{-1} (r)$. Then
$Z_r = \phi^{-1} (r)/\Z_m$  is a principal $S^1/\Z_m$-bundle over $M_r$:
\begin{equation}\label{orbundle}
S^1/\Z_m \hookrightarrow  Z_r \to  M_r.
\end{equation}
This is not always a bundle of smooth manifolds. For our purposes,
it is sufficient to consider it as a principal circle bundle in the
category of topological spaces. Let $\Z$ be the integral lattice of
$S^1$, then $\Lambda'=\frac 1m\, \Z \subset \R$ is the integral
lattice of $S^1/\Z_m$, i.e., $S^1/\Z_m\cong \R/\Lambda'$. Let
$e(Z_r)\in H^2 (M_r; \Lambda')$ be the Euler class of the bundle
(\ref{orbundle}). As a topological invariant, $e(Z_r)$ is constant
as a function over $I$, hence we will simply use $e$ to denote this
class.

\begin{theorem}\label{DH} (Duistermaat--Heckman \cite{DH})
 Let $a, b\in I$, where $I$ is a connected  open interval of regular values of a proper moment map $\phi$ of a   Hamiltonian circle action.  Let $\omega_a$ and $\omega_b$ be respectively the reduced symplectic forms in
 $M_a$ and in $M_b$ and $[\omega_a]$ and $[\omega_b]$  be the cohomology classes they represent.
 Then in $H^2(M_r;\R)$ we have
 $$[\omega_b] - [\omega_a] = e \cdot (b-a).$$
 Here, $e\in H^2(M_r; \Lambda')$  is the Euler class of (\ref{orbundle}), and we used the canonical identification
 of $H^2(M_a)$,  $H^2(M_b)$ and of  $H^2(M_r)$.
\end{theorem}

For the linear action (\ref{eqaction}) of $S^1$ with weight vector $w$
on $\C^n$ with the standard symplectic form, the moment map is
$$\phi = \frac12 \left( w_1 |z_1|^2 + \cdots + w_{k+1} |z_{k+1}|^2 \right),$$  and any
nonzero value of $\phi$ is a regular value. Since $S^{2k+1}$ is
$S^1$-equivariantly diffeomorphic to a regular level set of $\phi$,
$\CP^k_w$ is a symplectic reduced space at a regular value of $\phi$. Hence, we are in
the situation as described  in the Duistermaat--Heckman theorem. In
particular we obtain the following lemma.

\begin{lemma}\label{wpj}
Let $S^1$ act on $\C^{k+1}$ ($k > 0$) with weight vector $w = (w_1,
\dots, w_{k+1})$, where $w_i \in \N$ for $i=1, \dots, k+1$, and
$\gcd(w_1, \dots, w_{k+1}) =1$. Let $S^{2k+1}\subset \C^{k+1}$ and
let $\CP^k_{w} =S^{2k+1}/S^1$ be the weighted projective space. Let
$m=\mbox{lcm} (w_1, \dots, w_{k+1})$, and let $e$ be the Euler class
of the principal bundle
\begin{equation}\label{es}
S^1/\Z_m\hookrightarrow S^{2k+1}/\Z_m  \to \CP^k_{w}.
\end{equation}
Then in $H^2(\CP^k_{w}; \R)$ we have
$$e  = \pm \,\frac{1}{m} t,$$
where $t\in H^2(\CP^k_{w}; \Z)$ is a generator.
\end{lemma}

\begin{proof}
Let us consider first the bundle as a principal $S^1=\R/\Z$-bundle
via the isomorphism $S^1\to S^1/\Z_m$,
and compute its Euler class in $H^2(\CP^k_{w};\Z)$.
Kawasaki's computation \cite{Ka} of the integral cohomology of $S^{2k+1}/\Z_m$
gives $H^2(S^{2k+1}/\Z_m;\Z)=0$. In the Gysin sequence
$$ H^0(\CP^k_{w};\Z)\to H^2(\CP^k_{w};\Z)\to H^2(S^{2k+1}/\Z_m;\Z)$$
the first map sends 1 to the Euler class and must be surjective.
Thus the Euler class is a generator of $H^2(\CP^k_{w};\Z)$.
Considering our original $S^1/\Z_m=\R/\Lambda'$-bundle, it follows that the Euler class $e$ of (\ref{es}) defines a
generator in $H^2(\CP^k_{w}; \Lambda')$.  Here $\Lambda'$ is the
lattice of $S^1/\Z_m$. As elements of $H^2(\CP^k_{w}; \R)$, $e$ is
$\frac{1}{m}$ times a generator $\pm t$ of $H^2(\CP^k_{w}; \Z)$, where
$\Z$ is the lattice of $S^1$. Hence, we have our conclusion.
\end{proof}

\subsection{Consequences of our assumptions}
\
\medskip

In this section, we look at several essential ingredients for the proof of Theorem~\ref{A}
which are consequences of our assumptions. \\

First, as we saw in Theorem~\ref{LT}, the circle action was classified in \cite{LT} into two cases. Now we
state the two cases in terms of weights of the $S^1$ action on the normal
bundles of the fixed sets, $X$ and $Y$, where $\phi(X) < \phi(Y)$.
\begin{lemma}\label{action}
Under Assumption~\ref{conditions}, $S^1$ acts on the
 normal bundles $N_X$ of $X$ and $N_Y$ of $Y$  as in the following two cases:
\par
 ({\bf A}) $S^1$ acts on the fiber of $N_X$ with  weights $(1, \dots, 1)$, and acts on the fiber of $N_Y$
 with weights $(-1, \dots, -1)$;
\par
 ({\bf B}) $\dim(X)=\dim(Y)$, and $S^1$ acts on the fiber of $N_X$ with weights $(2,
\dots, 2, 1)$ and acts on the fiber of $N_Y$
 with weights $(-2, \dots, -2, -1)$.
\end{lemma}

Let  $\dim(X)=2i$ and $\dim(Y)=2j$.  Since $2j +2 =\dim(M)-\dim(X)$,
the normal bundle $N_X$ of $X$ has complex rank $j+1$. Hence $N_X$
is $S^1$-equivariantly diffeomorphic to a $\C^{j+1}$ bundle over
$X$, with the circle acting on $N_X$ fiberwise on $\C^{j+1}$. By
Lemma~\ref{action}, $S^1$ acts on $\C^{j+1}$ with either weight vector
$w=(1, \dots, 1)$ or weight vector $w=(2, \dots, 2, 1)$. Let $r$ be
a fixed regular value of $\phi$, then $\phi^{-1}(r)$ is
$S^1$-equivariantly diffeomorphic to an $S^{2j+1}$ bundle over $X$,
and hence $M_r =\phi^{-1}(r)/S^1$ is diffeomorphic to a $\CP^j_w$
bundle over $X$. Similarly, by looking at the normal bundle $N_Y$
of $Y$, $M_r$ is  diffeomorphic
to  a $\CP^i_{w'}$ bundle over $Y$.

Now, let us consider the $S^1$-orbibundle
$$S^1 \hookrightarrow \phi^{-1}(r) \to M_r.$$
If we restrict this bundle to the fiber $\CP^j_w$ of $M_r$, viewed
as a bundle over $X$,  we
 have  the Hopf fibration or the ``weighted Hopf fibration''
$$ S^1  \hookrightarrow  S^{2j+1}\to  \CP^j_w. $$
 Lemma~\ref{e(P)} below follows from Lemma~\ref{wpj}.
  \begin{lemma}\label{e(P)}
  In the case of Lemma~\ref{action}, let $\dim(Y)=2j$,
  let $w=(w_1, \dots, w_{j+1})$ be the weight
  vector of the $S^1$ action on $N_X$
  and let $m=\mbox{lcm} (w_1, \dots, w_{j+1})$.
 For a fixed value $r\in\left(\phi(X), \phi(Y)\right)$,
 let $e$ be the Euler class of the principal bundle
 \begin{equation}\label{bundle}
  S^1/\Z_m\hookrightarrow \phi^{-1}(r)/\Z_m\to M_r.
 \end{equation}
   Then
  $$e \,|_{\CP^j_w} = \pm \,t \quad\mbox{in case {\bf (A)},  and}
  \quad e \,|_{\CP^j_w}= \pm \,\frac{1}{2}\, t  \quad\mbox{in case {\bf (B)}},$$
   where
  $t\in H^2(\CP^j_w; \Z)$ is a generator.
   \end{lemma}

We will also need the integral cohomology ring of $Y$ and the value
$\phi(Y)-\phi(X)$. See Propositions 6.1 and 8.15 and Lemmas 8.3 and 8.4 in \cite{LT}.

\begin{lemma}\label{XY}
Under Assumption~\ref{conditions}, assume in addition that $[\omega]\in H^2(M; \R)$ is a primitive
integral class. Then
$$ H^*(Y;\Z) = \Z[v]/v^{j  + 1},\quad\mbox{where $v = [\omega|_Y] $ and  $2j=\dim(Y) $}.$$
Furthermore, $\phi(Y)-\phi(X) =1$ when the action is semifree; and
$\phi(Y)-\phi(X) = 2$ otherwise.
 \end{lemma}

\subsection{Proof of Theorem~\ref{A}}
\
\medskip

Now, we are ready to prove Theorem~\ref{A}.

\begin{proof} [Proof of Theorem~\ref{A}]
First of all, by the assumption on $X$ and $Y$ and the compactness of $M$, $M$ has to be connected.

Next, we show simply connectedness.  Since by Theorem~\ref{LT} we have $H^2(M; \Z) = \Z$, then, modulo
rescaling, we may assume that $[\omega]$ is a primitive integral
class.

Let $a$ and $b$ be any two values such that $\phi(X) < a < b < \phi(Y)$. Let  $\omega_a$ and $\omega_b$ be
the reduced symplectic forms in $M_a$ and in $M_b$ respectively.   By Theorem~\ref{DH},
\begin{equation}\label{DH'}
[\omega_b] - [\omega_a] = e \cdot (b - a),
\end{equation}
where $e\in H^2(M_r)$ is the Euler class as in Lemma~\ref{e(P)}. Let
$2i=\dim(X)$ and $2j=\dim(Y)$. Since $M_r$ is a bundle over $X$ and
is also a bundle over $Y$, let
$$p_1\colon M_r\to X \quad \mbox{and} \quad p_2\colon M_r\to  Y $$
be the projection maps.     Let $u=[\omega|_X]$ and $v =
[\omega|_Y]$. By continuity of the classes of the reduced symplectic
forms,
$$p_1^*(u) = \lim_{a\to \phi(X)} [\omega_a] \quad\mbox{and}\quad
p_2^*(v) = \lim_{b\to \phi(Y)} [\omega_b].$$ Taking limit in
(\ref{DH'}), we obtain
\begin{equation}\label{DH''}
p_2^*(v) - p_1^*(u) = e \cdot \left(\phi(Y)-\phi(X)\right).
\end{equation}
Restricting (\ref{DH''}) to the fiber $\CP^j_w$ of $M_r$, viewed as
a bundle over $X$,  we get
$$ p_2^*(v)\,|_{\CP^j_w} = e\, |_{\CP^j_w} \cdot \left(\phi(Y)-\phi(X)\right).$$
By Lemma~\ref{e(P)} on $e
\,|_{\CP^j_w}$ and Lemma~\ref{XY}  on  $\phi(Y)-\phi(X)$,  we get
$$ p_2^*(v)\,|_{\CP^j_w} = \pm \,t,$$
where $t\in H^2\left(\CP^j_w; \Z\right)$ is a generator.  Hence
\begin{equation}\label{p2}
 p_2^*\left(v^j \right)\,|_{\CP^j_w} =\pm \,t^j.
\end{equation}
  Let
\begin{equation}\label{iotap}
f= p_2\circ  \iota\colon \,\, \CP^j_w
\stackrel{\iota}\hookrightarrow M_r\stackrel{p_2}\longrightarrow Y.
\end{equation}
 Then (\ref{p2}) is
$$ f^* (v^j) = \pm \,t^j.$$
By Lemma~\ref{XY}, $v^j\in H^{2j}(Y; \Z)$ is a generator. By
Lemmas~\ref{action} and \ref{221}, $H^*\left(\CP^j_w; \Z\right)=
\Z[t]/ t^{j+1}$, in particular, $t^j\in H^{2j} (\CP^j_w; \Z)$ is a
generator. Moreover, Lemmas~\ref{action} and \ref{221} also imply
that $\pi_1\left(\CP^j_w\right) = 1$. Now, Lemma~\ref{orc} allows us
to conclude
$$\pi_1(Y) = 1.$$ We can similarly prove $$\pi_1(X) = 1.$$ Decompose
$$M=\phi^{-1}\left(-\infty,r +\epsilon\right) \cup \phi^{-1}\left(r -\epsilon, +\infty\right).$$
Clearly, $\phi^{-1}\left(-\infty,r +\epsilon\right)$ deformation
retracts to $X$, and $\phi^{-1}\left(r - \epsilon, +\infty\right)$
deformation retracts to $Y$. Moreover, the intersection of the two
open sets is homotopy equivalent to $\phi^{-1}(r)$, which is
connected (we saw that it is diffeomorphic to an $S^{2j+1}$-bundle
over $X$). By the Seifert--van Kampen theorem,
 $$\pi_1(M)\cong \pi_1(X)*_{\pi_1\left(\phi^{-1}(r)\right)}\pi_1(Y)=1.$$
\end{proof}

\begin{remark}
For the last step of the proof, one can also refer to \cite{L}, where the first author proves that, for every Hamiltonian circle action on a connected compact symplectic manifold $(M,\omega)$ one has
$$\pi_1(M)\cong \pi_1\left(\mbox{minimum of  $\phi$}\right)\cong \pi_1\left(\mbox{maximum of  $\phi$}\right).$$
\end{remark}

\section{Proof of Theorem~\ref{hcp}}\label{corolsection}

We need Lemma~\ref{homotopyclassification} below to prove Theorem~\ref{hcp}.
\begin{lemma}\label{homotopyclassification}
Let $M$ be a compact simply-connected $2n$-dimensional manifold with cohomology ring
$H^*(M; \Z) \cong {\Z}[x]/(x^{n+1})$. Then
there is a homotopy equivalence  $\Tilde f\colon M\to \CP^n$ such that
$ \Tilde f^* (a) = x $, where $a\in H^2\left(\CP^n; \Z\right)$ is a generator.
\end{lemma}

\begin{proof}
Since $\CP^{\infty}$ is a $K(\Z, 2)$ space, there is up to homotopy a unique map $f\colon  M \to\CP^{\infty}$
such that $f^*(t)=x$, where $t\in H^2(\CP^{\infty};\Z)$ is the standard generator \cite[Theorem 4.57]{H}. Since $\dim
M=2n$, we may assume that $f$ factors through $\Tilde f$:
$$
\xymatrix {
& \CP^n \ar[d]^i \\
M \ar[ur]^{\Tilde f} \ar[r]_f & \CP^{\infty}. }
$$
Letting $a=i^*(t)$,  then $H^*(\CP^n;\Z)=\Z[a]/(a^{n+1})$. Since
$\Tilde f^* i^*=f^*$, it follows that $\Tilde f^*(a)=x$ and that
$\Tilde f^*$ is an isomorphism. Since all homology groups of $M$ and $\CP^n$ are finitely generated, the universal coefficient theorem shows that 
$\Tilde f$ also induces an isomorphism on all (integral) homology groups. Since $M$ and $\CP^n$ are simply
connected, by the Whitehead Theorem \cite[Corollary 4.33]{H},  the map $\Tilde f$ is a homotopy
equivalence. (Here we use that every topological manifold has the
homotopy type of a CW-complex. See \cite[Theorem V.1.6]{B}.)
\end{proof}

We recall that a homotopy complex projective space is said to have standard
Pontryagin classes if there exists a homotopy equivalence to $\CP^n$
which preserves the Pontryagin classes of the tangent bundles.

\begin{proof}[Proof of Theorem~\ref{hcp}]
By Theorem~\ref{A},  $X$, $Y$ and $M$ are  one-connected. The claim
for $X$ follows from (\ref{hcx}) and
Lemma~\ref{homotopyclassification}. The claim for $Y$ and for $M$
follows similarly by using (\ref{hcy}) and (\ref{ma}). In both cases the total Chern class is standard, so the total Pontryagin class is standard as well.
\end{proof}

\section{Finiteness up to $S^1$-equivariant diffeomorphism --- Proof of Theorem~\ref{equiv.diffeo}}
\label{eqthmsection}

In this section, we prove Theorem~\ref{equiv.diffeo}. We only need to restrict attention to the case when $\dim(M)=2n\geq 6$ since, as mentioned in the Introduction, when $\dim(M)=2$ or $4$, the manifold $M$ is unique up to equivariant symplectomorphism.

\subsection{Reducing the proof of Theorem ~\ref{equiv.diffeo} to an equivariant pseudo-isotopy classification} \label{51assumptions}
\
\medskip

In this section, we will see that we can reduce the proof of Theorem~\ref{equiv.diffeo} to the classification of equivariant pseudo-isotopy classes of equivariant self-diffeomorphisms of the sphere bundle of an extremum of the moment map.
We mostly use geometric topology techniques. So we assume that:
\begin{itemize}
\item
$M$ is a $2n$-dimensional compact smooth manifold with a smooth $S^1$-action,
\item
the fixed point set has two connected components $X$ and $Y$, with $\dim(X)=2i$ and $\dim(Y)=2(n-i-1)$, and
\item there is a Morse-Bott function $\phi\colon M\to\R$ with critical set $X\cup Y$. 
\end{itemize}
In the following paragraph, we summarize the facts we know and we will use about $M, X, Y$ and the circle action.
We use $\S$ to denote the sphere bundle of the normal bundle of $X$ and we denote by $\P$ the quotient $\S / S^1$.

The manifold $M$ is simply connected, and $X$ and $Y$ are homotopy complex projective spaces with standard Chern classes (hence also standard Pontryagin classes) (Theorem~\ref{A} and Theorem~\ref{hcp}).
Moreover, by Theorem~\ref{LT}, we have the following two cases.

\noindent Case $\bf (A)$, the action is semi-free. In this case, the fiber of the equivariant normal bundle of $X$ is an $(n-i)$-dimensional complex free $S^1$-representation space, so $S^1$ acts on $\S$ freely, and the quotient $\P$ is a smooth $\CP^{n-i-1}$-bundle over $X$. Moreover, the Chern classes of the normal bundles of $X$ and of $Y$ are uniquely determined.

\noindent Case $\bf (B)$, the action is not semifree. In this case, $\dim(X) = \dim(Y) = 2i$, $n=2i+1\geq 3$ is odd.
The equivariant normal bundle of $X$ splits into a direct sum $N_1\oplus N_2$, where the fibers of $N_1$ and of $N_2$ are respectively of complex dimensions $1$ and $n-i-1$. The Chern classes of $N_1$ and of $N_2$ are uniquely determined.
The circle acts on the fiber of $N_1$ with weight $1$, and acts on the fiber of $N_2$ with weights
$(2, \cdots, 2)$ (see Lemma~\ref{action}). So the sphere bundle $\S$ has two connected smooth strata:
the bottom stratum $\S_0=S(N_2)$ consisting of all points with stabilizer group $\Z_2$,
and the top stratum $\S\setminus \S_0$ consisting of all points with trivial stabilizer group.
The quotient $\P =\S_0/S^1\cup \big((\S\setminus \S_0)/S^1\big) =\P_0\cup (\P\setminus \P_0)$ is a bundle over $X$ with fiber  $\CP^{n-i-1}_{(2,\dots , 2,1)}$. By Lemma~\ref{221}, this fiber is homeomorphic to $\CP^{n-i-1}$. More precisely, the pair $(\P,\P_0)$ fibers over $X$ with fiber pair homeomorphic to $(\CP^{n-i-1}, \CP^{n-i-2})$. So $\P$ and $\P_0$ are topological projective bundles over a homotopy complex  projective space $X$; in particular, $\P\setminus \P_0$ is a bundle over $X$ with contractible fiber $\CP^{n-i-1}\setminus \CP^{n-i-2}$, so it is homotopy equivalent to $X$.
Note that the spaces $\P,\P_0$ and $\P\setminus\P_0$ are simply-connected, and
$\dim(\P_0) = 2n-4$ and $\dim(\P\setminus\P_0) = 2n-2$.\\

 First, we have finiteness up to equivariant diffeomorphism of the tubular neighborhoods of $X$ and $Y$:
\begin{lemma}
 If Assumption~\ref{conditions} holds, then up to $S^1$-equivariant diffeomorphism, tubular neighborhoods of $X$ and of $Y$ are determined up to finite ambiguity.
\end{lemma}

\begin{proof}
 Since $X$ and $Y$ are homotopy complex projective spaces with standard Pontryagin classes, by
Proposition~\ref{finite-hcp}, there are finitely many diffeomorphism types of $X$ and $Y$ if they are not 4-dimensional.
If $X$ or $Y$ is 4-dimensional, since it is a symplectic homotopy $\CP^2$ with standard 1st Chern class, by a theorem of Taubes \cite[Corollary 7.2]{Sm}, it is symplectomorphic to the standard $\CP^2$. (This is the only place in Section~\ref{eqthmsection} where we directly use the symplectic structure.)

By Theorem \ref{LT}, the Chern classes of the normal bundles of  $X$ and of $Y$ are uniquely determined
in the above two cases. Then Lemma~\ref{finite-bundle} implies that the normal bundles of $X$ and of $Y$ are determined up to finite ambiguity. Since the circle acts in the above two ways, we have finitely many possibilities up to equivariant diffeomorphism for the tubular neighbourhoods of $X$ and $Y$.
\end{proof}

Since the Morse-Bott function $\phi$ has no more critical sets, the manifold $M$ is obtained by gluing the tubular neighbourhoods of $X$ and $Y$.
Hence to prove Theorem ~\ref{equiv.diffeo}, it is enough to show that there are only finitely many ``different'' ways to glue the two tubular neighborhoods. A gluing is an equivariant diffeomorphism $\varphi$ between the sphere bundle
$\mathbb S_X$ of $X$ and the sphere bundle $\mathbb S_Y$ of $Y$. The result of the gluing is the manifold $\mathbb
D_X\cup_\varphi \mathbb D_Y$ obtained from the disjoint union of the
disk bundles $\mathbb D_X$ and $\mathbb D_Y$ of $X$ and $Y$, (identified with their tubular neighborhoods) by identifying the two boundaries via $\varphi$.
Hence, to prove Theorem~\ref{equiv.diffeo},  we only need to prove
Propositions~\ref{W} and \ref{pseudoisotopy finiteness} below.

\begin{definition}
Let $Z,Z'$ be two manifolds and let $\varphi_0,\varphi_1\colon Z\to Z'$ be two diffeomorphisms. A {\em pseudo-isotopy} between $\varphi_0$ and $\varphi_1$
is a diffeomorphism $\Phi\colon Z\times I\to Z'\times I$ such that $\Phi(z,0)=(\varphi_0(z),0)$ and
$\Phi(z,1)=(\varphi_1(z),1)$ for $z\in Z$.
If $Z, Z'$ are $S^1$-manifolds, and $\varphi_0,\varphi_1\colon Z\to Z'$ are two $S^1$-equivariant
diffeomorphisms,  an {\em $S^1$-equivariant pseudo-isotopy}  between $\varphi_0$ and $\varphi_1$
is an $S^1$-equivariant diffeomorphism $\Phi\colon Z\times I\to Z'\times I$ such that $\Phi(z,0)=(\varphi_0(z),0)$ and $\Phi(z,1)=(\varphi_1(z),1)$ for $z\in Z$. Here, $S^1$ acts on the first factors on $Z\times I$
and on $Z'\times I$.
\end{definition}

Proposition~\ref{W}  below was essentially proved by  K. Wang, and
we refer to \cite[Proposition 1.1]{Wa1}.   The proof applies to both
semifree and non-semifree actions.
\begin{proposition}\label{W}
Let $\varphi_0,\varphi_1\colon\S_X\to \S_Y$ be two equivariantly pseudo-isotopic diffeomorphisms. Then there exists
an equivariant diffeomorphism $$\mathbb D_X\cup_{\varphi_0} \mathbb D_Y \cong \mathbb D_X\cup_{\varphi_1} \mathbb D_Y$$
between the results of the two gluings. 
\end{proposition}
Since any two diffeomorphisms between the sphere bundles differ by a
self-diffeomorphism of one of the sphere bundles, we consider instead the group of pseudo-isotopy classes of
equivariant self-diffeomorphisms of one of the sphere bundles, say, over $X$, the minimum of $\phi$.

\begin{proposition}\label{pseudoisotopy finiteness}
The group of equivariant pseudo-isotopy classes of equivariant diffeomorphisms $\S\to \S$ is finite when
$\dim(M) = 2n \ge 6$.
\end{proposition}
We will prove Proposition~\ref{pseudoisotopy finiteness} in the following three subsections.

\begin{remark}
For us, the existence of a decomposition $\mathbb D_X\cup_{\varphi} \mathbb D_Y \cong M$ is sufficient.
Hausmann and Holm show that given the manifold $M$ with its symplectic structure and circle action, one can control the choices 
in the tubular neigbourhood embeddings $\mathbb D_X\to M$, $\mathbb D_Y \to M$ so that the diffeomorphism 
$\varphi$ is determined by $M$ up to isotopy and certain gauge transformations \cite[Lemma 4.3]{HH}.
\end{remark}

\subsection{The case when $\dim(M)=2n > 6$ and the action is semifree}
\
\medskip

The proof of Proposition~\ref{pseudoisotopy finiteness} for this case was essentially given in K. Wang's
paper \cite{Wa1}.  Because of some minor inaccuracies in Wang's
proof,  and since we want to extend  the argument to
non-semifree actions, we will give a description  of  his whole argument
in this section.

\bigskip

When the $S^1$ action is semifree,  an equivariant self-diffeomorphism of $\S$ induces a self-diffeomorphism of the smooth quotient $\P$.
The map $\pi\colon \S\to \P$ is a principal $S^1$-bundle and gives rise to an Euler class in $H^2(\P)$.
A self-diffeomorphism of $\P$ induced by an equivariant diffeomorphism of $\S$ preserves the Euler class.

On the other hand,  a diffeomorphism  $\P\to\P$ which preserves the Euler class can be lifted to an
equivariant  diffeomorphism $\S\to \S$. The lift is not unique, but two such lifts $f_0,f_1$
differ by an element $g\in \mbox{Maps}(\P, S^1)$, i.e.,  for $x\in \S$, one has $f_1(x)=f_0(x)\cdot g(\pi(x))$.
Since $\P$ is simply-connected, there is a homotopy $H\colon \P\times I \to S^1$ from $g$ to the constant map.
Then $\psi \colon \S\times I \to \S\times I$,
defined by $\psi (x,t)=\left(f_0(x)\cdot H\left(\pi(x), t\right),t\right)$ is an isotopy between $f_0$ and $f_1$.
Thus a self-diffeomorphism of $\P$ which preserves the Euler class can be lifted to a unique pseudo-isotopy
class of equivariant  self-diffeomorphisms of $\S$.

\medskip

  In the following statements, we will always assume the circle action to be semifree ($\P$ is smooth), we will however
  indicate the dimension of the manifold for the statements to hold.

\begin{lemma}\label{idonH}
Let $\dim(M)=2n$. For any $n$, in the group of self-diffeomorphisms of $\P$ which preserve the Euler class $e$ of the circle bundle $\S\to \P$, the subgroup of self-diffeomorphisms inducing the
identity on $H^*(\P;\Z)$ has finite index (which equals 1 or 2).
\end{lemma}

\begin{proof}
Suppose $\P$ is a projective bundle over a homotopy complex projective space of dimension $2i$. Let $x\in H^2(\P;\Z)$ be the image of the generator of $H^2$ of the
homotopy complex projective space. Then by the theorem of the projective bundle,
$$H^*(\P;\Z)\cong \Z[e,x]/\left\langle x^{i+1},p(e,x) \right\rangle,$$
where the polynomial $p$ is defined by $p(e,x)=\sum_k e^{n-i-k} (-1)^k c_k x^k$ for some $c_k\in \Z$.
Thus a $\Z$-basis of $H^*(\P;\Z)$ is given by the monomials $e^l x^k$ with $0\le l \le n-i-1$, $0\le k\le i$.
Any diffeomorphism $f$ must induce an automorphism of $H^2(\P;\Z)$.
Since $f^*(e)=e$, we have $f^*(x)=\pm x + \lambda e$ for some $\lambda\in \Z$.
Up to maybe restricting to a subgroup of index 2, we may assume $f^*(x)= x + \lambda e$.

It follows that in $H^*(\P;\Z)$ we have $0=f^*(x^{i+1})=f^*(x)^{i+1}=(x + \lambda e)^{i+1}$
and similarly $0=p(e,x+\lambda e)$. In the polynomial ring $\Z[e,x]$, we obtain the equality
$$\left\langle x^{i+1},p(e,x) \right\rangle=\left\langle (x+\lambda e)^{i+1},p(e,x+\lambda
e) \right\rangle.$$
We leave it to the reader to prove that this implies $\lambda=0$. (One can consider the cases
$\deg(p)>i+1$, $\deg(p)=i+1$, $\deg(p)<i+1$.)
It follows that $f^*=id$.
\end{proof}

\begin{lemma}\label{homotopy finiteness}
Let $\dim(M)=2n$. For any $n$, the group of homotopy classes of self-homotopy equivalences of $\P$ inducing the identity on cohomology is finite.
\end{lemma}
K. Wang refers to an argument of Kahn, but we give a short proof here
using obstruction theory. Recall its setting \cite[pages 415ff]{H}:
Given two CW-complexes $A,B$, and a map from a subcomplex $A_0$ of $A$ to
$B$, does this map extend to a map $A\to B$? For simplicity we
assume $B$ is simply-connected. Obstruction theory says such
an extension exists if a sequence of obstruction classes
\begin{equation}\label{obstr}
o_j\in H^{j+1}\big(A,A_0;\pi_j(B)\big) \quad \mbox{for $j=2,3,\dots$}
\end{equation}
vanishes. However, only the first obstruction (the $o_j$ where $j$ is minimal such that $\pi_j(B)\ne 0$)
is well-defined and depends only on the map $A_0\to B$; the higher obstructions $o_j$ are only defined if the lower
obstructions vanish, and they depend on choices. Let $A^{(j)}$ denote the $j$-skeleton of $A$.
Then $o_j$ depends on the choice of an extension of $A_0\to B$
to $ A_0\cup A^{(j-1)}$ which can be extended to $A_0\cup A^{(j)}$;
if the corresponding obstruction $o_j$ vanishes, then it can be extended to $A_0\cup A^{(j+1)}$.

\begin{proof}
Let $\mathcal H=\mathcal H(\P)$ denote the group of homotopy classes of self-homotopy equivalences of $\P$ inducing the identity on cohomology. For each $j$, let $\mathcal H_j$ be the group of those homotopy classes of self-homotopy equivalences $\varphi\colon \P\to \P$
inducing the identity on cohomology, such that there exists a homotopy between the restrictions of $\mbox{id}$ and $\varphi$
to $\P^{(j)}$. Obviously $\mathcal H_{j}\subseteq \mathcal H_{j-1}$ for all $j$, and $\mathcal H_{\dim(\P)}$ is the trivial group.

Obstructions for a homotopy between two maps $\P\to\P$ lie in $H^j\big(\P;\pi_j(\P)\big)$.
(We have $A=\P\times I, A_0=\P\times\partial I, B=\P$ in the notation above.)
Since the maps are equal on $H^2\left(\P; \Z\right)$, the first obstruction in $H^2\big(\P;\pi_2(\P)\big)$ is 0,
and one has a homotopy between the restriction of the maps to $\P^{(2)}$. Thus $\mathcal H=\mathcal H_2$.
The statement of the lemma will follow from the statement that for all $j>2$, the index of
$\mathcal H_{j}$ in $\mathcal H_{j-1}$ is finite, which we will now prove.

For $j>2$ the group $H^j\big(\P;\pi_j(\P)\big)$ is finite since the
cohomology of $\P$ is concentrated in even degrees and the other even
homotopy groups of $\P$ are finite. (This uses the long exact sequences in
homotopy for $S^1\to S^{2k+1}\to \CP^k$ and $\CP^k\to \P \to \CP^j$.)
There is no well-defined map $\mathcal H_{j-1}\to H^j\big(\P;\pi_j(\P)\big)$,
since the obstruction to the existence of a homotopy over $\P^{(j)}$
depends on the choice of a homotopy
over $\P^{(j-2)}$ which extends to $\P^{(j-1)}$.
However, we can associate to each element $\varphi\in \mathcal H_{j-1}$
the subset $\mathcal O_j(\varphi)\subseteq H^j\big(\P;\pi_j(\P)\big)$
consisting of all obstruction classes for all such choices.
Then $\mathcal H_j=\left\{ \varphi\in \mathcal H_{j-1} \mid 0\in \mathcal O_j(\varphi)\right\}$.

Now assume that $o_j\in \mathcal O_j(\varphi_1) \cap \mathcal O_j(\varphi_2)$,
corresponding to homotopies $h_1$ respectively $h_2$ between the
restrictions of $id$ and $\varphi_1$ respectively $\varphi_2$.
Then the homotopies $h_1$ and $h_2$ can be combined to a homotopy $h$
between the restrictions of
$\varphi_1$ and $\varphi_2$ to $\P^{(j-2)}$ which extends to $\P^{(j-1)}$, and the obstruction for extending $h$ to $\P^{(j)}$ is 0.
This follows from an additivity formula for the obstruction classes, see \cite[Theorem 4.2.7]{Ba}.
By composing with $\varphi_1^{-1}$ we see that $0\in \mathcal O_j(\varphi_1^{-1}\varphi_2)$ and that $\varphi_1^{-1}\varphi_2\in \mathcal H_j$.
We have shown that if $o_j\in \mathcal O_j(\varphi_1) \cap \mathcal O_j(\varphi_2)$, then $\varphi_1$ and $\varphi_2$ are contained
in the same coset modulo $\mathcal H_j$.
This shows that in $\mathcal H_{j-1}$ there are at most as many cosets modulo $\mathcal H_j$ as elements in $H^j\big(\P;\pi_j(\P)\big)$, so that
$\mathcal H_{j}\subseteq \mathcal H_{j-1}$ has finite index.
\end{proof}

\begin{proposition}\label{D0}
For  $\dim(\P)=2n-2 > 4$,  the group $\mathcal D_0(\P)$ of pseudo-isotopy classes of self-diffeomorphisms of $\P$
which are homotopic to the identity is finite.
\end{proposition}
\begin{proof}
The argument involves the surgery exact sequence, and is parallel to \cite{Wa2}, Sections 1 and 2.
For the surgery exact sequence in general the reader may consult \cite{Wall} or \cite{Lueck}.

There is a map from the so-called structure set
$S(\P\times I,\partial)$ to $\mathcal D_0(\P)$ (in fact a group homomorphism):
an element in the structure set is represented by a homotopy equivalence
of manifolds with boundary $H\colon K\to \P\times I$ which restricts to a
diffeomorphism of the boundaries. Since $\dim(\P)>4$, we can apply the $s$-cobordism theorem:
we may assume that $K=\P\times I$, and furthermore we may assume that
$H|_{\P\times\{0\} }$ is the identity.
Now the map $S(\P\times I,\partial)\to \mathcal D_0(\P)$ is defined by sending
$H\mapsto\ H|_{\P\times\{1\} }$.
K. Wang shows that the map is well-defined and surjective \cite{Wa2}.

We also need the following surgery obstruction groups: the group $L_{2n-1}(1)$ is trivial,
and $L_{2n}(1)$ is cyclic of infinite order or order 2 (depending on whether $n$ is even or odd).
Moreover $[\Sigma \P, G/O]$ (respectively $[\Sigma^2 \P, G/O]$)
is the group of homotopy classes of maps from the suspension of $\P$
(respectively double suspension) to a certain space $G/O$.

For $\dim(\P) =2n-2 > 4$ there is a diagram:
\begin{equation}\label{sequence1}
\xymatrix@C-7pt{
[\Sigma^2 \P, G/O]
\ar[r]
&
L_{2n}(1)
\ar[r]
&
S(\P\times I,\partial)
\ar[r]
\ar[d]
&
[\Sigma \P, G/O]
\ar[r]
&
L_{2n-1}(1)=0
\\
& &
\mathcal D_0(\P)
.}
\end{equation}
Here the first row is the surgery exact sequence: since the spaces we consider
are products with an interval, this is in our case an exact sequence of groups. We use that $\P$ is simply connected (see beginning of Section \ref{51assumptions}).

By a theorem of Novikov \cite[Theorem 2.18]{MM}, every element in the image
of $L_{2n}(1)\to S(\P\times I,\partial)$ is of the form
$$\P \times I\, \# \,\Sigma^{2n-1} \stackrel{id\# c}{\longrightarrow}\P \times I\, \#\, S^{2n-1} = \P \times I,$$
where $\#$ is connected sum, $\Sigma^{2n-1}$ is an exotic sphere, and $c\colon \Sigma^{2n-1}\to S^{2n-1}$ is a homeomorphism.
Since the group of exotic spheres is finite in dimension at least 5,
the image of $L_{2n}(1)$ in $S(\P\times I,\partial)$ is finite.

We need the following facts about $G/O$: it is an infinite loop space with finite odd homotopy groups,
and $\pi_{4k}(G/O)\otimes \Q\cong \Q$. See for example \cite[p.215]{BV}, \cite[Theorem 6.48]{Lueck},
\cite{MM}.
Since $\Sigma \P$ has vanishing even reduced cohomology groups,
the Atiyah--Hirzebruch spectral sequence shows that $[\Sigma \P, G/O]$ is finite.
It follows that $S(\P\times I,\partial)$ and $\mathcal D_0(\P)$ are also finite.
\end{proof}

\begin{proof}[Proof of Proposition~\ref{pseudoisotopy finiteness}] (for the case when $\dim(M)=2n > 6$ and the action is semifree.)
In this case $S^1$ acts
freely on $\S$,  and we saw in the beginning of this subsection that pseudo-isotopy classes of equivariant
self-diffeomorphisms of $\S$ correspond exactly to pseudo-isotopy classes of self-diffeomorphisms of $\P$ which preserve the
Euler class  of the bundle $\S\to\P$.
Let $\mathcal D(\P)$ be the group of pseudo-isotopy classes of self-diffeomorphisms of $\P$ which
act as identity on $H^*(\P; \Z)$.  By Lemma~\ref{idonH}, it suffices to prove that $\mathcal D(\P)$
is finite. Let $\mathcal H(\P)$ be the group of homotopy classes of
self-homotopy equivalences of $\P$ which act as  identity on  $H^*(\P; \Z)$.
We have a natural map $\varphi\colon \mathcal D(\P)\to \mathcal H(\P)$ with $\ker(\varphi)= \mathcal D_0(\P)$.
By Lemma~\ref{homotopy finiteness} and Proposition~\ref{D0}, the groups $\mathcal H(\P)$ and $\ker(\varphi)$ are finite, thus $\mathcal D(\P)$
is finite.
\end{proof}

\subsection{The case when $\dim(M)=2n > 6$ and the action is not semifree}
\
\medskip

In this section, we prove Proposition~\ref{pseudoisotopy finiteness} for the case when $\dim(M)=2n > 6$ and the action is not semifree. Since $n$ is odd in this case,
we have $\dim(M)=2n \geq 10$. For the case of semifree
actions,  we worked with maps between the quotients of the group
actions,  and we used the $s$-cobordism theorem in the proof of
Proposition~\ref{D0}.  For the case of non-semifree actions,  the quotients are not canonically
smooth manifolds, so it is more convenient not to work with the
quotients but to consider equivariant maps instead.
However the $s$-cobordism theorem does not generalize to the
category of all equivariant maps, but only to the smaller category
of isovariant maps, which we describe below.  The appropriate
generalization of the surgery theory argument to the case of several
strata in Proposition~\ref{ds1} below uses the Browder--Quinn isovariant
surgery exact sequence. The original source for the isovariant surgery theory is the article by Browder and Quinn \cite{BQ};
Weinberger \cite{We} gives an outline of the proofs of the statements; some aspects are dicussed in great detail by Dovermann and Schultz \cite{DS}.

An equivariant map is {\em isovariant} if it preserves stabilizer groups of points;  it is {\em transverse linear isovariant}
if it is isovariant and is transverse to the lower strata. Every equivariant diffeomorphism is transverse linear isovariant.
(In the one stratum case, all equivariant maps are transverse linear isovariant.)

In the following, a homotopy equivalence $f\colon S_1\to S_2$ is a homotopy equivalence in the category
of transverse linear isovariant maps, i.e., $f$ is transverse linear isovariant,  there exists $g\colon S_2\to S_1$
which is transverse linear isovariant,  and,  there exist  transverse linear isovariant homotopies such that
$g\circ f\simeq \mbox{id}$ and $f\circ g\simeq\mbox{id}$.

\begin{remark}
For the general Browder-Quinn theory, one has to restrict further to maps whose restrictions to strata of dimension less than $5$ are equivariant diffeomorphisms, and homotopies whose restrictions to such strata are constant. In our case
when $\dim(M)=2n \geq 10$, the bottom stratum $\S_0$ of $\S$ has dimension at least $7$, and the bottom
stratum $\P_0$ of $\P$ has dimension at least $6$, so these restrictions do not apply.
\end{remark}

Corresponding to this case, our surgery argument which replaces Proposition~\ref{D0} concerns the group $\mathcal D^{S^1}_0(\S)$ of equivariant pseudo-isotopy classes of self-diffeomorphisms of $\S$ which are
transverse linear isovariantly homotopic to the identity.

\begin{proposition}\label{ds1}
When $\dim(M)=2n \geq 10$, $n$ is odd, the group $\mathcal D^{S^1}_0(\S)$ is finite.
\end{proposition}
\begin{proof}
Note that since $\dim(M)=2n \geq 10$, we have $\dim(\S)=2n-1\geq 9$.

Again we start by defining the various sets occurring in the isovariant surgery exact sequence:
$[\Sigma \P, G/O]$ is the group of homotopy classes of maps from the suspension of $\P$ to $G/O$ as before, so it is a finite set as before.
The set $S^{S^1}_0(\S\times I,\partial)$ consists of homotopy equivalences
$(N,\partial N)\to (\S\times I,\partial)$ in the category of transverse linear isovariant maps, modulo
concordance.
Since $\dim(\S)>5$, we can apply the isovariant $s$-cobordism theorem \cite{BQ}: $N$ is equivariantly diffeomorphic to a cylinder,
and the element is represented by a homotopy equivalence $(\S\times I,\partial)\to (\S\times I,\partial)$
which is the identity on $\S\times{0}$.
So again we get a surjective map $S^{S^1}_0(\S\times I,\partial)\to \mathcal D^{S^1}_0(\S)$.
The abelian obstruction groups $L^{S^1}_{j}(\S\times I,\partial)$
can be computed to be:
$L^{S^1}_{2n+1}(\S\times I,\partial)\cong \Z\oplus \Z_2$,
$L^{S^1}_{2n}(\S\times I,\partial)\cong 0$.
One uses an exact sequence from \cite{BQ} which relates
them to the $L$-groups of the strata, which are isomorphic to the usual
non-equivariant $L$-groups of the quotients. These quotients $\P_0$ and $\P \setminus \P_0$ are simply-connected (see the beginning of Section \ref{51assumptions}).
In particular there is the short exact sequence
$$
\xymatrix@C-10pt{
0
\ar[r]
&
L^{S^1}_{2n+1}\left((\S\setminus \S_0)\times I,\partial\right)
\ar[d]^\cong
\ar[r]
&
L^{S^1}_{2n+1}(\S\times I,\partial)
\ar[r]
&
L^{S^1}_{2n-1}(\S_0\times I,\partial)
\ar[d]^\cong
\ar[r]
&
0
\\
&
L_{2n}(1)\cong\Z_2
& &
L_{2n-2}(1)\cong \Z.
}
$$
In the above, we used the fact that $n\geq 5$ is odd
for our case of non-semifree actions. 

The isovariant surgery exact sequence is the exact sequence of groups in the first row of the diagram
$$
\xymatrix@C-10pt{
[\Sigma^2 \P, G/O]
\ar[r]^<<<<\sigma
\ar[d]^{i^*}
&
L^{S^1}_{2n+1}(\S\times I,\partial)
\ar[r]
\ar[d]^{i^*}
&
S^{S^1}_0(\S\times I,\partial)
\ar[r]
\ar[d]
&
[\Sigma \P, G/O]
\ar[r]
&
L^{S^1}_{2n}(\S\times I,\partial)
\\
[\Sigma^2 \P_0, G/O]
\ar[r]^<<<\sigma
&
L^{S^1}_{2n-1}(\S_0\times I,\partial)
&
\mathcal D^{S^1}_0(\S).
}
$$
Finiteness of $S^{S^1}_0(\S\times I,\partial)$, and thus of $\mathcal D^{S^1}_0(\S)$ now follows from the following
lemma.
\end{proof}

\begin{lemma}\label{L-action}
The image of $L^{S^1}_{2n+1}(\S\times I,\partial)$ in $S^{S^1}_0(\S\times I,\partial)$ is finite.
\end{lemma}
\begin{proof}
We show that the image of $[\Sigma^2 \P, G/O]$ in $L^{S^1}_{2n+1}(\S\times I,\partial)$
is infinite. Since $L^{S^1}_{2n+1}(\S\times I,\partial)\to L^{S^1}_{2n-1}(\S_0\times I,\partial)$
has finite kernel, it suffices to show that the composition
$$[\Sigma^2 \P, G/O]\overset{i^*}\longrightarrow [\Sigma^2\P_0, G/O] \overset{\sigma}\longrightarrow L^{S^1}_{2n -1}(\S_0\times I,\partial)\cong \Z$$
has nontrivial image.
We can compute rationally:
\begin{eqnarray*}
[\Sigma^2 \P, G/O] \otimes \Q & \cong & \bigoplus_j H^j\left(\Sigma^2 \P;\pi_j G/O\otimes \Q\right),\\
\ [ \Sigma^2 \P_0, G/O ]
\otimes \Q & \cong & \bigoplus_j H^j\left(\Sigma^2 \P_0;\pi_j G/O\otimes \Q\right).
\end{eqnarray*}
The map $i^*$ corresponds to the map induced by inclusion
$$\bigoplus_j H^j\left(\Sigma^2 \P;\pi_j G/O\otimes \Q\right)
\to \bigoplus_j H^j\left(\Sigma^2 \P_0;\pi_j G/O\otimes \Q\right).$$
Since $\P_0$ is a sub-projective bundle of $\P$, we can apply the Leray--Hirsch theorem
to show that $i^*$ is rationally surjective.
The  map $\sigma$ is the map in the ordinary surgery exact sequence for $\P_0$, see (\ref{sequence1}).
We saw that this map must be rationally surjective (using Novikov's theorem).
Thus the lemma follows.
\end{proof}

Finally  Lemmas \ref{idonH} and \ref{homotopy finiteness} are replaced by the following result.
\begin{proposition}\label{transverse homotopy}
For $\dim(M) = 2n > 6$, $n$ odd,  the subgroup of equivariant diffeomorphisms $\psi\colon \S\to \S$ such that $\psi$ is transverse linearly isovariantly homotopic to the identity has finite index in the group of all equivariant self-diffeomorphisms of $\S$.
\end{proposition}
\begin{proof}
Let $\varphi\colon \S\to \S$ be any equivariant diffeomorphism.
We try to construct a transverse linear isovariant homotopy from $\varphi$ to $\mbox{id}$ in three steps.
We will see that in the first and third step, we need to restrict to those $\varphi$ which are in a certain subgroup
of finite index to construct the homotopy.

{\bf Step 1}: We construct a homotopy on $\S_0$.
There is a restriction homomorphism $r\colon \mbox{Diff}^{S^1}\!\! (\S) \to \mbox{Diff}^{S^1}\!\! (\S_0)$
from the group of equivariant self-diffeomorphisms of $\S$ to the group of equivariant self-diffeomorphisms of $\S_0$.
By Lemmas \ref{idonH} and \ref{homotopy finiteness},
the normal subgroup of $\mbox{Diff}^{S^1}\!\! (\S_0)$ consisting of those diffeomorphisms which are equivariantly homotopic to the identity has finite index.
It follows that its inverse image under $r$ has finite index in $\mbox{Diff}^{S^1}\!\! (\S)$.
For $\varphi$ contained in this subgroup, we let
$h\colon \S_0\times I \to \S_0$ be an equivariant homotopy between $\mbox{id}$ and $r(\varphi)$, which we may assume to be constant for $t\ge \frac12$,
i.e. $h(x,t)=\varphi(x)$ for $x\in \S_0$ and $t\ge \frac 12$.

{\bf Step 2}: We extend the homotopy on $\S_0$ to a transverse linear isovariant map from a tubular neighbourhood
of $\S_0\times I$ to $\S$. This can be done by Lemma~\ref{step2}.

{\bf Step 3}: We extend the map from the neighbourhood of $\S_0\times I$ to a transverse linear isovariant
homotopy $h'\colon \S\times I \to \S$ such that $h'(x, 0)= x$ and $h'(x, 1) = \varphi(x)$. By Lemma~\ref{step3}, there are finitely many obstructions with values in finite groups against doing this.
Again these obstructions depend on some choices (for example choices in steps 1 and 2), but an argument similar to the proof of
Lemma \ref{homotopy finiteness} shows that step 3 can be done for $\varphi$ contained in a subgroup of finite index.
\end{proof}

\begin{lemma}\label{step2}
Let $\varphi\colon \S\to \S$ be an equivariant diffeomorphism.
Let $h\colon \S_0\times I \to \S_0$ be an equivariant homotopy from the identity to the
restriction of $\varphi$ to $\S_0$ such that $h(x, t) = \varphi(x)$ for $x\in \S_0$ and $t\ge \frac12$.
Then we can extend $h$ to a transverse linear isovariant map from a tubular neighbourhood of
$\S_0\times I$ in $\S\times I$ to $\S$.
\end{lemma}

\begin{proof}
In order to make the notation less complicated, we replace $I$ by $[0,2]=I\cup[1,2]$ and assume that the homotopy $h\colon \S_0\times [0,2] \to \S_0$ is constant
for $t\ge 1$.

The codimension 2 invariant submanifolds $\S_0\subset \S$ and $\S_0\times I \subset \S\times I$ have equivariant normal bundles with total spaces
$\nu$ and $\nu\times I$. By definition the fiber of $\nu$ at $x\in \S_0$ is $T_x\S/T_x\S_0$.
Let $g$ be an $S^1$-invariant metric on $\S$.
Then $g_t=(1-t)g+t\varphi^*g$ is an $S^1$-invariant metric on $\S\times \{t\}$.
The metric $g$ identifies the fiber of $\nu$ at $x$ with the orthogonal complement (w.r.t $g$) of $T_x\S_0$ in $T_x\S$,
thus defines a scalar product on the fibers of $\nu$ (by restriction of $g$ to this orthogonal complement). We can thus use
the disk bundle $D_\epsilon^{g}(\nu)$ of vectors of norm at most $\epsilon$. We do similarly for $\nu\times I$,
where in the $t$-slice we use the metric $g_t$.

Let $\epsilon$ be small such that the exponential map of $g_t$ in the normal direction in each slice
gives tubular neighbourhoods $N^{g}_\epsilon(\S_0)$ and $N^{g_t}_\epsilon(\S_0\times I)$
(the union of the $\epsilon$-neighbourhoods w.r.t $g_t$ in the $t$-slices) and identifications
$\Psi\colon  N^{g}_\epsilon(\S_0) \cong D^g_\epsilon(\nu)$ and
$\psi\colon  N^{g_t}_\epsilon(\S_0\times I) \cong D^{g_t}_\epsilon(\nu \times I).$

Now if we pull back $\nu$ via $h\colon \S_0\times I\to \S_0$, the
pull-back is isomorphic as equivariant vector bundle to $\nu\times
I$. The reason is that $h$ is equivariantly homotopic to $\mbox{pr}_1\colon
\S_0\times I\to \S_0$, so that $h^*\nu\cong \mbox{pr}_1^*\nu=\nu\times I$.
Thus we get a bundle map $\nu\times I\to \nu$ covering $h$, which we
can assume to be equivariant and orthogonal on each fiber, so
it restricts to a map $D^{g_t}_\epsilon(\nu \times I)\to
D^g_\epsilon(\nu)$, which again restricts to an orthogonal map on
each fiber. Composing with $\psi$ and $\Psi^{-1}$, we get a map
$H\colon N^{g_t}_\epsilon(\S_0\times I)\to N^{g}_\epsilon(\S_0)$
such that $H(x,0)=x$ and $\Psi \big(H(x,1)\big)$ differs from
$\Psi\big(\varphi(x)\big)$ by an equivariant orthogonal bundle automorphism
of $\nu$.

This is an orientable 2-dimensional Euclidean vector bundle, so the
group of equivariant bundle automorphisms is isomorphic to
$\mbox{Maps}\big(\S_0/S^1,SO(2)\big)$. But since $\S_0/S^1$ is
simply-connected, there is a homotopy to the constant map. This
produces a homotopy $N^{\varphi^*(g)}_\epsilon(\S_0)\times [1,2] \to
N^g_\epsilon(\S_0)$ from the restriction of $H$ to $\varphi$. The
union of this map and $H$ is a transverse linear isovariant map $f$
from a tubular neighbourhood of $\S_0\times [0,2]$ to a tubular
neighbourhood of $\S_0$. This maps boundaries to boundaries and
restricts to the identity in the slice $t=0$, and to $\varphi$ in
the slice $t=2$.
\end{proof}

\begin{lemma}\label{step3}
Let $\dim(M)=2n > 6$, where $n$ is odd. Let $N(\S_0\times I)$ be a tubular neighbourhood of $\S_0\times I$ in $\S\times I$,
and $F\colon \S\times \partial I \cup N(\S_0\times I)\to \S$ be a transverse linear isovariant map.
Then there are finitely many obstructions with values in finite groups for extending $F$
to a transverse linear isovariant map $\S\times I \to \S$.
\end{lemma}
\begin{proof}
To ensure isovariance, it suffices to extend the restriction
$$f\colon \big(\S\times \partial I \cup N(\S_0 \times I)\big) \setminus (\S_0 \times I)  \to \S\setminus \S_0$$
of $F$ to an equivariant map $(\S\setminus \S_0)\times I \to \S\setminus \S_0$.
On the quotients, we need to extend a map $$\big(\P\times \partial I \cup  N(\P_0 \times I)\big) \setminus (\P_0 \times I)  \to \P\setminus \P_0$$
to  a map $(\P\setminus \P_0)\times I\to \P\setminus\P_0$.
By (\ref{obstr}), the obstructions for the extension  lie in
\begin{eqnarray*}
& &H^{j+1}\Big((\P\setminus \P_0)\times I,\big(\P\times \partial I \cup  N(\P_0 \times I)\big) \setminus (\P_0 \times I);\pi_j(\P\setminus \P_0)\Big)\\
& \cong & H^{j+1}\big(\P\times I,\P\times \partial I \cup  N(\P_0 \times I);\pi_j(\P\setminus \P_0) \big) \quad \text{by excision}\\
& \cong & H^{j+1}\left(\P\times I,\P\times \partial I \cup  \P_0 \times I;\pi_j(\CP^i) \right)\quad \text{by homotopy invariance}\\
& \cong & H^j\left(\P,\P_0;\pi_j(\CP^i)\right) \quad \text{by the K\"unneth theorem.}
\end{eqnarray*}
For $j=2$ or for odd $j$  or for even $j>\dim(\P)$, the obstruction groups
vanish since the pair $(\P,\P_0)$ has vanishing cohomology in these
dimensions (recall $\dim(\P_0) > \dim (X) =2i \geq 4$ when $\dim(M) =2n > 6$, $n$ odd).
For $2<j\le \dim(\P)$ even, the obstruction groups are finite since $\pi_j(\CP^i)$ is finite.

It remains to show that we can lift an extension $K\colon (\P\setminus \P_0)\times I\to \P\setminus \P_0$ on the quotients to an equivariant extension of $f$.

There exists a lift of the map $K\colon (\P\setminus \P_0)\times I\to \P\setminus \P_0$ to some equivariant map
$\Tilde{K}\colon  (\S\setminus \S_0)\times I \to \S\setminus \S_0$,
since $K$ preserves the Euler class of the circle bundles.
This can be seen by restricting to $(\P\setminus \P_0)\times \{ 0 \}$, where such a lift --- namely the restriction of $f$ --- exists.

The restriction of $\Tilde{K}$ to
$\big(\S\times \partial I \cup N(\S_0 \times I)\big) \setminus (\S_0 \times I)$
differs from $f$ by a map
$$
g\in \mbox{Maps}\Big(\big(\P\times \partial I \cup N(\P_0 \times I)\big)
\setminus (\P_0 \times I), S^1\Big),$$ i.e.,
$f(x,t)=\Tilde{K}(x,t)\cdot g\left(\pi(x,t)\right)$.

Since $\big(\P\times \partial I \cup  N(\P_0 \times I)\big) \setminus (\P_0 \times I)$
is simply-connected
(this can be proved by the Seifert--van Kampen theorem), the map $g$ factors
through the universal covering $\R\to S^1$.
By the Tietze extension theorem, $g$ extends to a map
$G\colon (\P\setminus \P_0) \times I\to \R\to S^1$.
Now $\bar{K}\colon (\S\setminus \S_0)\times I \to \S\setminus \S_0$, defined by
$\bar{K}(x,t)=\tilde{K}(x,t)\cdot G\left(\pi(x,t)\right)$ is another equivariant lift of
$K$, and it is the desired extension of $F$ to an equivariant map
$(\S\setminus \S_0)\times I \to \S\setminus \S_0$.
\end{proof}

\begin{proof}[Proof of Proposition~\ref{pseudoisotopy finiteness}] (for the case when $\dim(M)=2n > 6$ and the action is not semifree.)
Let $\mathcal D^{S^1}(\S)$ be the group of
equivariant pseudo-isotopy classes of equivariant diffeomorphisms of $\S$.
Let $\mathcal H^{S^1}(\S)$ be the group of transverse linear
isovariant homotopy classes of transverse linear isovariant self-homotopy equivalences of $\S$.
Again there is a natural map $\varphi\colon  \mathcal D^{S^1}(\S)\to \mathcal H^{S^1}(\S)$.
By Proposition~\ref{L-action} we have that $\ker(\varphi)=\mathcal D_0^{S^1}(\S)$ is finite,
and Proposition~\ref{transverse homotopy} shows that
$\mathcal H^{S^1}(\S)$ is finite. Thus $\mathcal D^{S^1}(\S)$ is finite.
\end{proof}

\subsection{The case when $\dim(M)=2n=6$}
\
\medskip

In this section, we prove Proposition~\ref{pseudoisotopy finiteness} for the case when $\dim(M)=2n=6$.
In this case the preceding surgery arguments do not work. However a pseudo-isotopy classification of diffeomorphisms of 4-manifolds treats diffeomorphisms of 5-dimensional manifolds (the 4-manifold times the interval), and more refined surgery arguments do work in dimension 5.
Before we prove the proposition in this case, let us make the following observation.

\begin{lemma}\label{n3q}
In the case of a non-semifree action and $dim(M) = 2n = 6$, the pair $(\P,\P_0)$ is homeomorphic to the pair consisting of $S^2\times S^2$ and its diagonal,
and the top stratum $\P\setminus \P_0$ is diffeomorphic to the total space $TS^2$ of the tangent bundle of the 2-sphere.
\end{lemma}
\begin{proof} In this case, the fixed component $X$ is diffeomorphic to $\CP^1$. By \cite[Theorem 2]{LT},
the equivariant normal bundle of $X$ in $M$ is the direct sum of two complex line bundles: $N_1\oplus N_2$, where $S^1$ acts on the fiber of $N_1$ with weight $1$ and acts on the fiber of $N_2$ with weight $2$, and $c_1(N_1)=u$ and $c_1(N_2)=0$ with $u\in H^2(X; \Z)$ being a generator.
As in Lemma~\ref{221}, there is an equivariant map $$N_1\oplus N_2\to N_1^{\otimes 2}\oplus N_2,\quad (v_1,v_2)\mapsto (v_1\otimes v_1,v_2)$$ between the total spaces of the two bundles, where $S^1$ acts on the fiber of $N_1^{\otimes 2}\oplus N_2$ with weights $(2, 2)$. It induces a homeomorphism between the pairs $(\P,\P_0)$ and $\left(\P\left(N_1^{\otimes 2}\oplus N_2\right), \P(N_2)\right)$.

Now the total space of the canonical bundle $N_1$ is the subspace $$\big\{\big((x_0,x_1),[y_0:y_1]\big)\in \C^2 \times\CP^1 \mid x_0y_1=x_1y_0 \big\},$$
similarly the total space of $N_1^{\otimes 2}$ is $$\big\{\big((x_0,x_1),[y_0:y_1]\big)\in \C^2 \times\CP^1 \mid x_0y_1^2=x_1y_0^2\big\},$$ and finally
$\P\left(N_1^{\otimes 2}\oplus N_2\right)$ is homeomorphic to $$\Sigma_2=\big\{\big([x_0:x_1:x_2],[y_0:y_1]\big)\in \CP^2 \times\CP^1 \mid x_0y_1^2=x_1y_0^2\big\}.$$
The subspace $\P(N_2)$ consists of the points where $x_0=x_1=0$.
Here we use the same notation as Hirzebruch in \cite{Hi}, who writes down the homeomorphism
\begin{align*}
\CP^1\times\CP^1&\to  \Sigma_2\\
\big([u_0:u_1],[v_0,v_1]\big)&\mapsto \left(\left[\frac{v_1^2(-\overline{v_1}u_0+\overline{v_0}u_1)}{|v_0|^2+ |v_1|^2} :\frac{v_0^2(-\overline{v_1}u_0+\overline{v_0}u_1)}{|v_0|^2+ |v_1|^2}:v_0u_0+v_1u_1  \right],[v_0:v_1]\right)
\end{align*}
Finally we compose with a self-map of $\CP^1\times\CP^1$ sending the diagonal to
the points where $-\overline{v_1}u_0+\overline{v_0}u_1=0$, namely $\big([u_0:u_1],[v_0:v_1]\big) \mapsto \big([u_0:u_1],[\overline{v_0}:\overline{v_1}]\big)$.

Note that the homeomorphism we constructed is smooth when restricted to each stratum, in particular it defines a diffeomorphism between
$\P\setminus\P_0$ and $S^2\times S^2\setminus \Delta$, where $\Delta$ denotes the diagonal. The latter space is diffeomorphic to the total space of $TS^2$ by stereographic projection:
for each $x \in S^2$, stereographic projection gives a bijection
between $S^2 \setminus \{ x \} $ and the tangent space to $S^2$ at $x$. Composing with this gives the required diffeomorphism.
\end{proof}

\medskip
\begin{proof}[Proof of Proposition~\ref{pseudoisotopy finiteness}] (for the case when $\dim(M)=2n=6$.)
In the case when the action is semifree, $\P$ is a smooth simply connected closed $4$-manifold. By a theorem of Kreck \cite[Theorem 1]{K2}, all self-diffeomorphisms of $\P$ inducing the identity on cohomology are pseudo-isotopic to the identity. Together with Lemma~\ref{idonH}, this implies that the group of pseudo-isotopy classes of self-diffeomorphisms of $\P$ preserving the Euler class of the circle bundle $\S\to\P$ is finite. Hence so is the group of pseudo-isotopy classes of equivariant self-diffeomorphisms of $\S$.

Now we consider the case when the action is not semifree. We prove the proposition in the following steps.

{\bf Step 0}: We restrict to a subgroup of finite index.

We start with an equivariant self-diffeomorphism $\Phi\colon\S\to\S$.
The quotient map $\overline{\Phi}\colon\P\to\P$ induces self-diffeomorphisms of the strata $\P_0$ and $\P\setminus\P_0$.
The strata $\P_0$ and $\P\setminus\P_0$ are both homotopy equivalent to $S^2$, so self-diffeomorphisms of the strata act as $\pm \mbox{id}$ on $H_2(\P_0)$ and $H_2\big(\P\setminus\P_0\big)$. If both signs are positive,
then $\overline{\Phi}$ induces the identity on the homology of $\P$, since $(\P,\P_0,\P\setminus\P_0)$ is homotopy equivalent to $(S^2\times S^2,\text{diagonal, antidiagonal)}$ by Lemma~\ref{n3q}.
So the subgroup of equivariant self-diffeomorphisms of $\S$ such that the quotient map $\overline{\Phi}\colon\P\to\P$ induces the identity on homology has index at most 4 in the group of all
equivariant self-diffeomorphisms of $\S$.

In the following we show that an equivariant self-diffeomorphism $\Phi\colon\S\to\S$ in this subgroup is pseudo-isotopic to the identity.

{\bf Step 1}:
We show that the restriction of $\Phi$ to $\S_0$ is equivariantly isotopic to the identity.

Let $\varphi\colon\S_0\to\S_0$ be the restriction of $\Phi$,
inducing a self-diffeomorphism $\overline{\varphi}\colon\P_0\to\P_0$ on the quotient.

Since $\P_0$ is diffeomorphic to $S^2$, by a theorem of Munkres \cite{M}, every orientation-preserving self-diffeomorphism $\overline{\varphi}$ is isotopic to the identity.
Let $\overline{\varphi}_t\colon I\times\P_0\to\P_0$ be such an isotopy, i.e. $\overline{\varphi}_0=\mbox{id}$ and $\overline{\varphi}_1=\overline{\varphi}$.

The pull-back of the $S^1$-bundle $\S_0\to\P_0$ by $\overline{\varphi}_t$
is isomorphic to $I\times\S_0$, so we obtain an $S^1$-equivariant map  $\varphi_t\colon I\times\S_0\to\S_0$
covering  $\overline{\varphi}_t$, such that $\varphi_0(x)=x$
for all $x$.  Since $\P_0$ is simply-connected, we may also assume that  $\varphi_1(x)=\varphi(x)$
for all $x$; so we obtain an equivariant isotopy $\varphi_t\colon I\times\S_0\to\S_0$ from the identity to the restriction $\varphi$ of $\Phi$.

{\bf Step 2}:
        We show that $\Phi$ is equivariantly isotopic to a self-diffeomorphism $\Phi_0\colon\S\to\S$ whose restriction to a neighborhood $N(\S_0)$ of $\S_0$ is the identity.

 The equivariant isotopy $\varphi_t$ is generated by an equivariant vector field $V(t,x)$ on $\S_0$. An equivariant extension of this vector field
produces an ambient isotopy $\Phi_t\colon I\times \S\to \S$, such that $\Phi_1=\Phi$, and $\Phi_0$ is the identity on $\S_0$.
Moreover, by uniqueness of equivariant neighbourhoods up to isotopy (see Theorem VI.2.6 in \cite{B}), we may assume that the restriction of $\Phi_0$ to a small equivariant tubular neighbourhood of $\S_0$ is an equivariant vector bundle automorphism. The group of equivariant normal bundle automorphisms of $\S_0$ is isomorphic to $\mbox{Maps}\big(\P_0,SO(2)\big)$, and since $\P_0$ is simply-connected, each such automorphism is homotopic to the trivial automorphism. It follows that we may assume that $\Phi_0$ is the identity on an equivariant tubular neighbourhood $N(\S_0)$ of $\S_0$.

{\bf Step 3}:
We show that every self-diffeomorphism $\Phi_0$ of $\S$ which restricts to the identity on $N(\S_0)$ is equivariantly pseudo-isotopic to the identity of $\S$. It suffices to construct an equivariant  pseudo-isotopy on $\S\setminus N(\S_0)$ relative to its boundary.

On $\S\setminus N(\S_0)$, the circle acts freely, so we can work with quotient spaces.
We can assume that the equivariant diffeomorphism $\Phi_0\colon \S\setminus N(\S_0)\to \S\setminus N(\S_0)$ is the identity on a neighbourhood of the boundary, so the induced quotient diffeomorphism
\begin{equation}\label{p-p0}
\psi=\overline{\Phi}_0\colon\P\setminus N(\P_0)\to \P\setminus N(\P_0)
\end{equation}
where $N(\P_0) = N(\S_0)/S^1$, is the identity on a neighborhood of the boundary.
Moreover,  $\psi$ induces identity on homology. Lemma \ref{4dboundary} below shows that the proof of \cite[Theorem 1]{K2} for simply-connected closed 4-manifolds can be slightly modified to show that $\psi$ is pseudo-isotopic to the identity.

Again, the pair $\left(I\times \big(\P\setminus N(\P_0)\big),\partial\right)$  is 1-connected, so the isotopy from $\psi=\overline{\Phi}_0$ to the identity lifts to an equivariant isotopy
on $\S\setminus N(\S_0)$ between $\Phi_0$ and the identity, constant on the boundary. This finishes the proof of the proposition.
\end{proof}

\begin{lemma}\label{4dboundary}
The diffeomorphism $\psi$ in (\ref{p-p0}) is pseudo-isotopic (relative boundary) to the identity.
\end{lemma}

\begin{proof}
By Lemma \ref{n3q}, the pair $\big(\P \setminus N(\P_0),\partial\big)$ is diffeomorphic to $\big(D(TS^2),S(TS^2)\big)$, the pair  of  (disk bundle, sphere bundle) of the tangent bundle of $S^2$. Let $Q=D(TS^2)$, which is a simply connected spin 4-manifold  with connected boundary. We identify $\psi$ in (\ref{p-p0}) with
 $$\psi\colon Q\to Q,$$
 which is a diffeomorphism restricting to the identity on the boundary, and inducing the identity on homology.

We consider $R=Q\times I$, this has boundary $\partial R=Q\times \{ 0 \} \cup \partial Q \times I \cup Q \times \{ 1\}$.
Given the self-diffeomorphism $f=\psi \cup \mbox{id} \cup \mbox{id}$ of $\partial R$, we ask whether it extends to a self-diffeomorphism of $R$.
If it does, then this self-diffeomorphism is a pseudo-isotopy from $\psi$ to $\mbox{id}_Q$ relative to the boundary.

We use $f$ to produce the twisted double $R\cup_f -R$. This is homeomorphic to the union $T_\psi \cup \partial Q\times D^2$,
where $T_\psi=Q\times I / (0,x)\sim (1,\psi(x))$ is the mapping torus of $\psi$.
Now a Mayer-Vietoris sequence shows that $H^*(T_\psi)=H^*(Q\times S^1)$, and another Mayer-Vietoris sequence shows that  $R\cup_f -R$ is spin if
$T_\psi$ is (the second Stiefel-Whitney class $w_2$ of the former restricts to $w_2$ of the latter).
Now $T_\psi$ is spin since $Q$ is, here one can use the Wang sequence as in \cite{K2}.
Now the proof proceeds as in \cite{K2}: one has a spin nullbordism $W$ of $R\cup_f -R$, which one can turn into a relative $h$-cobordism
by surgery.

Note that $H_3(R\cup_f -R)\cong H_2(Q)$, as is seen by the Mayer-Vietoris sequence
for $R\cup_f -R=T_\psi \cup \partial Q\times D^2$, also we have $H_3(Q)=H^1(Q,\partial Q)=0$. These are the two assumptions for the proof of Theorem 1 in \cite{K2} which hold trivially for simply-connected closed 4-manifolds.
\end{proof}

\section{Discussion and a question}\label{discussion}

 As we have discussed, under Assumption \ref{conditions}, when $X$ or $Y$ is an isolated point, by Delzant's theorem \cite{D}, $M$ is equivariantly symplectomorphic to $\CP^n$ with a standard circle
action as in (\ref{standard})  with $j=0$. When $\dim(M) = 6$ and when the action is semifree,
by \cite{D} and by a result of Gonzalez \cite{G}, $M$ is  equivariantly symplectomorphic to
$\CP^3$ with a standard circle action.

 In the case Delzant studied, the regular symplectic quotients are smooth $\CP^{n-1}$ with
 a standard symplectic structure. In the case  when $\dim(M) = 6$ and  the action is semifree,
  the regular  symplectic  quotients are smooth $4$-manifolds. Symplectic topological methods for
 symplectic $4$-manifolds give certain nice results on
 the symplectic structure, so that  the rigidity condition
 of the regular symplectic quotients
proposed by Gonzalez \cite{G} is fulfilled. The rigidity property of
the regular symplectic quotient grants a unique gluing up to
symplectic isotopy of the symplectic tubular neighborhoods of $X$
and of $Y$ along a regular level set of the moment map. When the
dimensions of the symplectic quotients  are  bigger, it is hard to
determine if the quotients are rigid.

For the $6$-dimensional Hamiltonian $S^1$-manifolds McDuff studied
in \cite{Mc} (see the Introduction), she uses the rigidity criterion
and  resolution of singularities of the $4$-dimensional
symplectic quotients to prove uniqueness up to equivariant
symplectomorphism of the Hamiltonian $S^1$-manifolds. In her case,
the symplectic quotients are symplectic $4$-orbifolds which have
isolated singularities.  In our case,  when $\dim(M) = 6$ and when the
action is not semifree, the symplectic quotient is
 $\P_{2,1}(\C\oplus H)$, the weighted
 projective bundle of the direct sum of a trivial and the Hopf bundle over $S^2$.
 We saw in Lemma \ref{n3q} that it is homeomorphic to
 $S^2\times S^2$ with the diagonal as the singular set.
 The method in \cite{Mc} does not generalize directly to this case.

We do not know examples other than $\CP^n$ and $\Gt_2(\R^{n+2})$ satisfying Assumption \ref{conditions}.
The following question is open.
\begin{question}
If Assumption \ref{conditions} holds, is $M$ $S^1$-equivariantly diffeomorphic
or symplectomorphic to $\CP^n$ or to $\Gt_2(\R^{n+2})$?
\end{question}

\begin{appendix}

\section{Non-equivariant results}\label{NonEqClassification}

In this appendix, we prove the non-equivariant uniqueness results mentioned in Remark~\ref{lowdim}. Due to the close relation of the arguments, we also prove finiteness of the manifold up to non-equivariant diffeomorphism. More precisely, we prove the following theorem.

\begin{theoremapp}\label{diffeo}
If Assumption \ref{conditions} holds, then
\begin{enumerate}
\item  up to diffeomorphism  there are finitely many such manifolds in each fixed
dimension;
\item  when $\dim M = 2n \leq 6$, $M$ is
diffeomorphic to $\CP^n$ if the action is semifree, and is
diffeomorphic to $\Gt_2(\R^{5})$ if the action is not semifree;
\item when $\dim M=10$ or $\dim M=14$ and the action is not semifree,
 $M$ is respectively homeomorphic to $\Gt_2(\R^{7})$ or to $\Gt_2(\R^{9})$.
\end{enumerate}
\end{theoremapp}

For the case of low dimensions and semifree actions, as mentioned in the Introduction, the results follow
from the literature. For the rest of the proof, we will use the one-connectivity of the manifold given by Theorem~\ref{A} and the integral cohomology ring and Chern classes of the manifold given by (\ref{ma}) and
(\ref{mb}) in Theorem~\ref{LT}. The more direct tool we use in this section is Kreck's modified surgery theory.

\subsection{The case when the action is semifree}\label{noneqA}
\
\medskip

 Theorem~\ref{diffeo} (1) for the case when $\dim(M)\neq 4$ and when the action is semifree
follows from Proposition~\ref{finite-hcp} below and Theorem~\ref{hcp}
(2).  For the cases when $\dim(M)\leq 6$ and when the action is semifree,
we saw in the Introduction that $M$ is diffeomorphic to $\CP^n$, where $n\leq 3$.

\begin{proposition} (\cite{Su, Lit})\label{finite-hcp}
Let $X$ be a homotopy complex projective space with standard Pontryagin classes.
If $\dim(X)\neq 4$, there are finitely many diffeomorphism types of $X$.
\end{proposition}

\subsection{The case when the action is not semifree}
\
\medskip

In this case a classification up to homotopy equivalence, which is
needed for the classical surgery theory used by Sullivan and Little,
is not known. We use Kreck's modified surgery theory \cite{K}, which
does not need the homotopy type, but the normal $(n-1)$-type (see
Definition~\ref{n-1type}) of the manifolds. We determine the normal
$(n-1)$-type of the manifold in Lemma~\ref{normaltype}, and then prove the theorem
using Kreck's theory and a computation by F. Fang and J. Wang.\\

\begin{lemma}\label{hypersurface}
The manifold $\Tilde{G}_2(\R^{n+2})$ is diffeomorphic to a quadratic
hypersurface in $\CP^{n+1}$, i.e., the vanishing set of a
homogeneous polynomial of degree 2.
\end{lemma}

\begin{proof}
Consider $\CP^{n+1}$ as the set of complex lines in the complexification of $\R^{n+2}$.
Define $\iota\colon\Tilde{G}_2(\R^{n+2})\to \CP^{n+1}$ such that it maps an oriented plane with oriented orthonormal basis $v,w$ to the complex line spanned by $v+iw$. It is easy to check that this map is well-defined, injective, smooth, and that the image is the vanishing set of the homogeneous quadratic polynomial $\sum z_i^2$.
\end{proof}

\begin{definition}\label{frakM_G}
Let $n\geq 3$ be odd. We name $\mathcal M_G$ to be any smooth,
compact, one-connected, almost complex $2n$-manifold with the same
integral cohomology ring and total Chern class as those of
$\Tilde{G}_2(\R^{n+2})$.
\end{definition}

Let $BO$ be the classifying space of the stable orthogonal group. It
is the union of the classifying spaces $BO_k$ of the orthogonal
groups, and each $BO_k$ is a union of Grassmann manifolds
$G_k(\R^{n+k})$ via the natural inclusions $G_k(\R^{n+k})\subseteq
G_k(\R^{n+k+1})$.

Let $N$ be a closed smooth $n$-manifold. An embedding $i\colon N \to \R^{n+k}$ gives rise to
a normal Gauss map $N\to G_k(\R^{n+k})$: to each $x\in N$ one assigns the normal space
of $i(N)$ at $i(x)$. In this way we also obtain a vector bundle over $N$, the normal
bundle of the embedding.

By composition we get the {\em{stable Gauss map}} $N\to BO$.
This map is (up to homotopy) independent of the choice of the embedding $i$. This is because the embeddings $i$ and $N\stackrel{i}{\to}\R^{n+k}\hookrightarrow \R^{n+k+1}$ have the same stable normal bundle map, two embeddings into $\R^{n+k}$ are isotopic for large $k$, and an isotopy of embeddings induces a homotopy of the stable Gauss maps.

The vector bundles over $N$ corresponding to isotopic embeddings into Euclidean space
are isomorphic, and the normal bundle of $N\stackrel{i}{\to}\R^{n+k}\hookrightarrow \R^{n+k+1}$ differs from the normal bundle of $i$ by the addition of a trivial rank 1 bundle. Thus every closed manifold $N$ has a unique stable normal bundle.

Recall that a stable (real or complex) vector bundle is an equivalence class of vector bundles, where we identify bundles which are isomorphic after adding trivial bundles to them.
For a compact Hausdorff space $X$, stable real (respectively complex) vector bundles over $X$ form a group $\Tilde{KO}^0(X)$ (respectively $\Tilde{K}^0(X)$) under direct sum, the reduced real (complex) $K$-theory of $X$. Stable real bundles over $X$ are classified
by homotopy classes of maps to $BO$. The direct sum of stable bundles induces a map
$\oplus\colon  BO\times BO\to BO$.

\begin{lemma}\label{chernclassstablevb}
Let $X$ be a compact Hausdorff space with the homotopy type of a finite $CW$-complex, and assume that $H^*(X; \Z)$ is torsion-free and concentrated in even degrees. Then two stable complex vector bundles over $X$ with the same total Chern class are isomorphic.
\end{lemma}

\begin{proof}
The Atiyah--Hirzebruch spectral sequence for reduced complex $K$-theory of $X$ degenerates since both the cohomology groups of $X$
and the coefficients of $K$-theory are zero in odd degrees. It follows that $\Tilde{K}^0(X)$ is a free abelian group. Thus the composition
$\Tilde{K}^0(X)\to \Tilde{K}^0(X)\otimes \Q \cong \bigoplus_i \Tilde{H}^{2i}(X;\Q)$ is injective. Here the isomorphism is given by the Chern character.
It follows that two stable vector bundles over $X$ with the same Chern classes are isomorphic.
\end{proof}

With an additional argument which covers the difference between stable and non-stable bundles, K. Wang proves the following result.
\begin{lemma} \cite [Proposition 3.1]{Wa1}\label{finite-bundle}
Let $X$ be a homotopy complex projective space.  Then there are at
most finitely many complex vector bundles of fixed rank over $X$
which have the same Chern classes.
\end{lemma}

\begin{definition}\label{n-1type}
Let $N$ be a $2n$-dimensional manifold. The {\em $n$-th Postnikov--Moore factorization} of
the stable Gauss map  $N\to BO$  is a factorization  $N\to B \to BO$ such that
$N\to B$ is {\em $n$-connected}, i.e., an isomorphism on the homotopy groups
$\pi_j$ for $j<n$ and surjective on $\pi_n$;  and $B \to BO$ is {\em $n$-coconnected}, i.e.,  an isomorphism on the homotopy groups $\pi_j$ for $j>n$ and injective on $\pi_n$.

We may assume that $B\to BO$ is a fibration. It is unique up to
fiber homotopy equivalence, and is called the {\em normal
$(n-1)$-type of $N$}.
\end{definition}

Let $H$ be the Hopf bundle on $\C \Pro$.
The virtual vector bundle (formal difference of vector bundles) $H^{\otimes 2}-(n+2)H$ defines a stable
 vector bundle over each $\CP^N$. The corresponding classifying maps are compatible and induce a (unique)
 classifying map $$\xi\colon\C\Pro \to BO.$$

The $n$-th Postnikov--Moore factorization of $\mbox{pt}\to BO$ is denoted by $\mbox{pt}\to BO\langle n+1\rangle \stackrel p\to BO$.
The space $BO\langle n+1\rangle$ is called the $(n+1)$-connected cover of $BO$.

\begin{lemma}\label{normaltype}
The normal $(n-1)$-type of $\mathcal M_G$ is $\C \Pro \times BO\langle n+1\rangle \stackrel {\xi\times p}\to  BO\times BO \stackrel \oplus \to BO$.
\end{lemma}

\begin{proof}
 If $\mathcal M_G =\Tilde{G}_2(\R^{n+2})$, then the statement is \cite[Proposition 3]{K}. We need to check that this remains true for all  $\mathcal M_G$.

Let $f\colon \mathcal M_G \to \C\Pro=K(\Z,2)$ represent the generator  $x\in H^2(\mathcal M_G; \Z)$.
Note that both $\mathcal M_G$ and  $\C\Pro$ are simply connected; using the integral cohomology ring
structures of these two spaces and the Hurewicz Theorem, we can check that  $f$ is $n$-connected.

Let $*$ be any constant map $\mathcal M_G \to BO\langle
n+1\rangle$.  We claim that the normal bundle map $\mathcal M_G\to
BO$ factors up to homotopy through
\begin{equation}\label{factor}
 \mathcal M_G \stackrel{(f,\,*)}\longrightarrow \C\Pro \times BO\langle n+1\rangle \stackrel {\xi\times p} \longrightarrow BO\times BO \stackrel \oplus \to BO.
 \end{equation}
To prove this, we need to show that the stable normal bundle of $\mathcal M_G $ is isomorphic to the pullback of the universal bundle via the composition
\begin{equation}\label{xif}
\mathcal M_G \stackrel {f} \to \C\Pro\stackrel {\xi} \to BO.
\end{equation}
Since $c\left(\mathcal M_G\right)= (1+x)^{n + 2}(1+2x)^{-1}$ (see
(\ref{mb})), the stable normal bundle  of  $\mathcal M_G$ has total
Chern class $(1+2x)(1+x)^{-(n+2)}$. Now the pullback of the
universal bundle via (\ref{xif}) is $f^*(-(n+2)H\oplus H^{\otimes
2})$. Hence the two stable bundles have isomorphic total Chern class
(induced by $f^*$). By Lemma~\ref{chernclassstablevb},
 the two stable bundles are isomorphic. It is easy to check
that (\ref{factor}) is a Postnikov--Moore factorization.
\end{proof}

For every fibration $B\to BO$, and every $k\in \N$, there is a
bordism group $\Omega_{k}(B)$. An element is represented by a closed
$k$-manifold $N$, together with a lift of its stable Gauss map along
$B \to BO$, but without any conditions on the homotopy groups. Two
such manifolds $N_0,N_1$ represent the same element if there is a
bordism between them, i.e. a compact $(k+1)$-manifold $W$ with
$\partial W = N_0 \cup N_1$, and $W$ itself is equipped with a lift
of its stable Gauss map to
 $B$ which restricts to the lifts of the
$N_i$'s. By a Pontryagin--Thom construction, the bordism group $\Omega_{k}(B)$
is the $k$-th homotopy group of the Thom spectrum of the stable vector bundle described
by the map $B\to BO$.

In particular we have the $2n$-th bordism group
corresponding to the fibration $\C \Pro \times BO\langle n+1 \rangle
\to BO$, and every manifold $\mathcal M_G$ (together with a choice
of lift) defines an element of this bordism group. Following F. Fang
and J. Wang \cite{FW}, we denote this bordism group by
$\Omega_{2n}^{\langle n+1\rangle} (\C\Pro;\xi)$.

In \cite{K}, Kreck uses his modified surgery theory to classify
manifolds with a given normal $(n-1)$-type. Kreck uses that two
manifolds which are diffeomorphic must in particular represent the
same element of the bordism group $\Omega_{k}(B)$, where $B$ is the
common normal $(n-1)$-type. On the other hand, if two manifolds
represent the same element in $\Omega_{k}(B)$, there is a bordism
between these two manifolds (which is additionally equipped with a
map to $B$), and one can try to modify this bordism by surgery in
the interior, so that the new bordism is an $s$-cobordism, which would
imply that the two manifolds are diffeomorphic. In general it is not
possible to perform such modifications, but in certain cases the
additional information contained in the map to $B$ implies that the
obstruction to do such surgeries vanishes. We have in particular Proposition~\ref{oddn}.

\begin{proposition}(\cite[Corollary 4]{K})\label{oddn}
Let $n\ge3$ be odd. Two closed, simply connected $2n$-manifolds with the
same Euler characteristic and the same normal $(n-1)$-type $B\to BO$
are diffeomorphic if they represent the same element in
$\Omega_{2n}(B)$.
\end{proposition}

\begin{proof}
[Proof of Theorem~\ref{diffeo}] (for  the case of non-semifree actions.)
 By Lemma~\ref{normaltype}, all $\mathcal M_G$ have the same
normal $(n-1)$-type, the normal $(n-1)$-type of
 $\Tilde{G}_2(\R^{n+2})$.  By Lemma~\ref{hypersurface},  $\Tilde{G}_2(\R^{n+2})$ is a hypersurface in $\CP^{n+1}$.
Thus we can use Kreck's modified surgery theory \cite{K}
and the computation by F.  Fang and J. Wang \cite{FW} for the normal $(n-1)$-type of a hypersurface.

By Proposition~\ref{oddn} and Lemma~\ref{normaltype}, two of the
manifolds $\mathcal M_G$ are diffeomorphic if they represent the
same element in $\Omega_{2n}^{\langle n+1\rangle} (\C\Pro;\xi)$.
More precisely, such an element is given by the map
$\mathcal M_G\stackrel{(f,\,*)}{\longrightarrow} \CP^\infty \times BO\langle
n+1\rangle$ from (\ref{factor}).

All our manifolds $\mathcal M_G$ have the same Pontryagin classes.
Using the integral cohomology ring structure of $\mathcal M_G$ given by (\ref{mb}), we can check that
\begin{equation}\label{degree}
\langle f^*(t)^n,[\mathcal M_G] \rangle=\langle
x^n,[\mathcal M_G]\rangle = 2.
\end{equation}
  In \cite{K}, p.745, for complete intersections, Kreck shows that
the Pontryagin classes and the
total degree determine the element in $\Omega_{2n}(B)\otimes \Q$.
The proof of this applies to our case, i.e.,
the Pontryagin classes and the number in
  (\ref{degree}) (which generalizes the total degree of a complete
intersection) determine the element in $\Omega_{2n}^{\langle n+1\rangle}
(\C\Pro;\xi)\otimes \Q$. Hence, the difference between two of the
manifolds $\mathcal M_G$ in $\Omega_{2n}^{\langle n+1\rangle}
(\C\Pro;\xi)$ is a torsion class. As a consequence, the size of the
finite torsion subgroup of this bordism group gives an upper bound
on the number of diffeomorphism types of the manifolds $\mathcal M_G$. This proves the finiteness result in part (1) of the theorem.

For $n=3$, the above torsion subgroup is trivial. This can be proved
in the following way. In this case, the Thom spectrum is $T\xi\wedge
M\mbox{Spin}$, where $T\xi$ is the Thom spectrum of $\xi$, see \cite{O},
section 1.4. Thus the bordism group is the sixth spin bordism group
of $T\xi$ and can be computed using the Atiyah--Hirzebruch spectral
sequence and the Thom isomorphism in singular homology. (For the
differential in the Atiyah--Hirzebruch spectral sequence one needs
to know, see \cite{T}, lemma on p. 751.) Thus $\mathcal M_G$ is
unique up to diffeomorphism, proving the 6-dimensional result in
part (2) of the theorem. We do not carry out the details of the
proof here, since this part of the theorem does also follow from
\v{Z}ubr's classification of simply-connected 6-manifolds \cite{Z}.

For $n=5$ and $n=7$, F. Fang and J. Wang prove in \cite[Propositions
4.3 and 5.1]{FW} that all torsion in $\Omega_{2n}^{\langle
n+1\rangle} (\C\Pro;\xi)$ is in the kernel of the map to
$\Omega_{2n}^{BTop\langle n+1\rangle} (\C\Pro;\xi)$. The latter is a
bordism group of topological manifolds and plays the corresponding
role in a homeomorphism classification of the manifolds. Part (3) of
the theorem follows from this.
\end{proof}
\end{appendix}

\end{document}